\documentclass[12pt,a4paper]{amsart}
\usepackage{amsmath}
\usepackage{amssymb}
\usepackage{amsthm}
\usepackage{amsrefs}
\usepackage{tikz}
\usepackage{microtype}
\usepackage{pdfsync}
\usepackage{url}

\usepackage[unicode]{hyperref}

\DeclareMathOperator{\dist}{dist}
\DeclareMathOperator{\supp}{supp}

\def\bbC{{\mathbb {C}}}
\def\bbR{{\mathbb {R}}}
\def\gc{\gtrsim}
\def\lc{\lesssim}

\theoremstyle{plain}
\newtheorem{theorem}{Theorem}[section]
\newtheorem{lemma}[theorem]{Lemma}

\newtheorem{prop}[theorem]{Proposition}
\newtheorem{corollary}[theorem]{Corollary}

\numberwithin{equation}{section}

\newtheorem{remark}[theorem]{Remark}

\theoremstyle{remark}

\newcommand{\bpar}{\big ( }
\newcommand{\parb}{\big ) }
\newcommand{\Bpar}{\Big ( }
\newcommand{\parB}{\Big ) }

\begin{document}
\date{\today}

\title
[Bilinear restriction estimates]
{Bilinear restriction estimates for surfaces of codimension bigger than one}

\author[J. Bak, J. Lee, S. Lee]
{Jong-Guk Bak, Jungjin Lee and Sanghyuk Lee}

\address {Department of Mathematics, Pohang University of Science and Technology, Pohang 37673, Republic of Korea}
\email{bak@postech.ac.kr}

\address{Department of Mathematical Sciences, School of Natural Science, Ulsan National Institute of Science and Technology, Ulsan 44919, Republic of Korea}
\email{jungjinlee@unist.ac.kr}

\address{Department of Mathematical Sciences, Seoul National University, Seoul 08826, Republic of Korea}
\email{shklee@snu.ac.kr}




\begin{abstract} In connection with the restriction problem in $\bbR^n$ for hypersurfaces including the sphere and paraboloid, the bilinear (adjoint) restriction estimates have been extensively studied.
However, not much is known about such estimates for surfaces with codimension (and dimension) larger than one. In this paper we show sharp bilinear $L^2 \times L^2 \rightarrow L^q$ restriction estimates for general surfaces of higher codimension. In some special cases, we can apply these results to obtain the corresponding linear estimates.
\end{abstract}

\maketitle

\section{Introduction and statement of results}
\label{intro}

For a smooth hypersurface $S$ such as the sphere or paraboloid in $\bbR^n$, $n\ge 3$, the $L^p$-$L^q$ boundedness of the (adjoint) restriction operator (or the extension operator) $\widehat{f d\sigma}$ {has} been extensively studied since the late 1960s. Here $d\sigma$ denotes the induced Lebesgue measure on $S$. Especially, when $S$ is the sphere, it is conjectured by E. M. Stein (cf. \cite{St}) that $\widehat{f d\sigma}$ should map $L^p (S)$ boundedly to $L^q (\bbR^n)$, precisely when $q \ge \frac{n+1}{n-1} p'$ and $q> \frac{2n}{n-1}$. Since then, a large amount of literature has been devoted to this problem. Over the last couple of decades, the bilinear and multilinear approaches have proven to be quite effective, and substantial progress has been made through those approaches. We refer the reader to \cites{BCT, BG, G} for the most recent developments.

On the other hand, when the dimension of the manifold is one, namely, when the associated surface is a curve, the restriction estimate is by now fairly well understood {\cites{BOS, BOS1, BOS2, Sto}}.

However, not much is known about the intermediate cases, namely, when the codimension $k$ of the manifold is {between} $1$ and $n-1$.  The restriction problem for quadratic surfaces of codimension $k\ge 2$ was first studied by Christ \cite{Ch1} and Mockenhaupt \cite{Moc}. They also considered  the  problem in a more general setting and found some necessary conditions {on the curvature} and codimension of the surface. For some surfaces they also established the  optimal $L^2 \rightarrow L^q$ linear estimates, which may be regarded as generalizations of the Stein-Tomas restriction theorem (see also \cite{B}). Although there are some known cases in which the $L^p$-$L^q$  boundedness is completely characterized (see for example \cites{BH, BL, ObR0}), for most surfaces with codimension bigger than one, the current state of the restriction problem is hardly beyond that of the Stein-Tomas theorem.

In this paper, we are concerned with restriction estimates for surfaces of codimension $k \ge 2$.
To be more specific, let us set $k \ge 1$ and  $I=[-1,1]$. Let $\Phi: I^d \rightarrow \mathbb R^k$ be a smooth function given by
\[
\Phi(\xi) = (\varphi_1(\xi), \varphi_2(\xi),\cdots, \varphi_k(\xi)).
\]
The adjoint restriction operator (the extension  operator) $E=E_\Phi$ for the surface
\(
(\xi, \Phi(\xi)) \in \mathbb R^d \times \mathbb R^k
\)
is defined by
\[
Ef(x,t) = \int_{I^d} e^{2 \pi i(x \cdot \xi + t\cdot \Phi(\xi))} f(\xi) d\xi,  \quad (x,t)\in  \mathbb R^d \times \mathbb R^k.
\]
{{}Specific examples of such operators  with $2\le k\le d-2$ can be found in \cites{BH, BL, Ch1, Moc, ObR0}. (Also, see Section 5.) }

There are some classes of surfaces for which  the optimal $L^2$-$L^q$ boundedness of $E$ is well understood.
In fact, using a Knapp type example  it is easy to see that $E$ may be bounded from $L^p$ to
$L^q$ only if $\frac{d+2k}{q}\le d\big(1-\frac1p\big)$. Hence, the best possible $L^2$-$L^q$ bound is that for $q = \frac{2(d+2k)}{d}$.
Christ \cite{Ch1} and Mockenhaupt \cite{Moc} showed that this is true for a class of surfaces satisfying a suitable curvature condition. In particular, let $M$ be a linear map from $\mathbb R^k$ to the space of $d\times d$ symmetric matrices and suppose that
$\int_{S^{k-1}} |\det M(t)|^{-\gamma} d\sigma(t)<\infty$
for $\gamma=\frac kd$. Then it was proven in \cite{Moc} that the extension operator $E$ defined by $\Phi=\xi^tM(t)\xi$
is bounded from  $L^2$ to $L^{\frac{2(d+2k)}{d}}$.

In order to obtain estimates for some $q<\frac{2(d+2k)}{d}$ and $p>2$, it seems necessary to consider methods other than the $TT^*$ argument which solely relies on the decay estimate for the Fourier transform of the surface measure. For this reason we wish to consider the bilinear restriction estimates for surfaces of codimension greater than 1 and try to obtain the best possible estimates.

\smallskip

Let $S_1,$ $S_2$ be closed cubes contained in $I^d$ and define
\[
E_if(x,t) = \int_{S_i} e^{2 \pi i(x \cdot \xi + t\cdot \Phi(\xi))} f(\xi) d\xi, \quad i=1,2.
\]
Let us consider the estimate
\begin{equation} \label{bilest}
\| E_1f \, E_2g \|_{L^q(\mathbb R^{d+k})}
\le C \|f\|_{L^p(\mathbb R^d)}\|g\|_{L^p(\mathbb R^{d})}.
\end{equation}

For the elliptic surfaces, bilinear estimates can be thought of {as a} generalization of linear estimates, since a linear restriction estimate follows from the corresponding bilinear one by an argument involving a Whitney decomposition. (See e.g. \cites{TVV}.)
The advantage of the bilinear estimates is that a wider rage of boundedness is possible than for the linear estimate, provided that a separation condition holds between the supports of the {functions $f$, $g$}. 
For surfaces with codimension 1, the sharp bilinear (adjoint) restriction estimate for the cone  was obtained by Wolff \cite{W3}, and for the paraboloid the corresponding estimate was proved by Tao \cite{T}. The bilinear approach has also been  applied to the restriction problem for hyperbolic surfaces: for the saddle surface in $\bbR^3$, Vargas \cite{V} and, independently,  Lee \cite{L} proved the bilinear  estimate by extending Tao's method.\footnote{For more general negatively curved surfaces in $\mathbb R^3$ and higher dimensions, Lee \cite{L} showed the bilinear restriction estimates.  However, in higher dimensions the linear estimate could not be deduced from the bilinear one, because the separation condition needed to prove the bilinear estimate for hyperbolic surfaces was more complex than that for the elliptic surfaces.} From these bilinear restriction estimates the corresponding linear ones have been obtained as well.

\smallskip
In order to state our results, we first introduce some notations.
For $\nu_1$, $\nu_2\in I^d$, we
define the $k\times d$ matrix $ \mathbf D(\nu_1,\nu_2)$ by
\[
\mathbf D(\nu_1,\nu_2)=
\begin{pmatrix}
\nabla \varphi_1(\nu_2) -  \nabla \varphi_1(\nu_1)\\
\vdots\\
\nabla \varphi_k(\nu_2) -  \nabla \varphi_k(\nu_1)
\end{pmatrix}.
\]
Here $\nabla\varphi_j$ is a row vector.
Let $H\varphi$ denote the Hessian of $\varphi$ and  $\mathbf D^t(\nu_1,\nu_2)$ be the transpose of $\mathbf D(\nu_1,\nu_2)$.
The following is our main theorem.

\begin{theorem} \label{thm}
Let $t=(t_1, \cdots ,t_k)$, $k \ge 1$.
Suppose that, for $\nu\in S_1\cup S_2 $ and  $|t|= 1$,
\begin{equation}
\label{hessianinvertible}
\det \Bpar \sum_{i=1}^k t_i H\varphi_i(\nu)\parB \neq 0
\end{equation}
and,  for $\nu_1 \in S_1$, $\nu_2 \in S_2$,  $|t|=1$ and for $\nu=\nu_1, \nu_2$,
\begin{equation}\label{hessiannonvanishing}
\det \Big[ \mathbf D(\nu_1,\nu_2)  \Bpar  \sum_{j=1}^k t_j H\varphi_j(\nu)
\parB^{-1}  \mathbf D^t(\nu_1, \nu_2)\Big] \neq 0.
\end{equation}
Then, for $q > \frac{d+3k}{d+k}$ and $\frac{1}{p}+\frac{d+3k}{d+k}\frac{1}{2q} < 1$,  the estimate \eqref{bilest} holds.
\end{theorem}

As special cases of Theorem \ref{thm}, one can deduce the known bilinear restriction theorems for the elliptic surfaces in \cite{T}
and the negatively curved ones in \cites{V,L}.

Let us set
\[
\mathbf M(t, \nu_1,\nu_2, \nu) : =
\begin{pmatrix}
0 & \mathbf D(\nu_1,\nu_2)\\
\mathbf D^t(\nu_1, \nu_2) & \sum_{i=1}^k t_i H\varphi_i(\nu) \end{pmatrix}.
\]
Assuming the condition \eqref{hessianinvertible}, it is easy to see  that \eqref{hessiannonvanishing} is equivalent to
\begin{equation}
\label{Mcondition}
\det \mathbf M(t, \nu_1, \nu_2,  \nu)\neq 0
\end{equation}
for $\nu_1 \in S_1$, $\nu_2 \in S_2$,  $|t|=1$ and for $\nu=\nu_1, \nu_2$.\footnote{One can use the block matrix formula $ \det \begin{pmatrix} A & B \\ C & D \end{pmatrix} = \det(D) \det(A-BD^{-1}C) $.}
The condition \eqref{Mcondition} may seem rather complicated, but such a condition appears naturally when one considers the bilinear $L^2\times L^2\to L^2$ estimate. When $k=1$, it is closely  related to the ``rotational curvature''. (See \cite{L} for more details.)
The necessity of the condition \eqref{Mcondition} will become clear in the course of the proof of Proposition \ref{prop:endEst} below.

From the condition  \eqref{hessiannonvanishing} it follows that the matrix $\mathbf D(\nu_1,\nu_2)$ has rank $k$. So, the vectors $\{\nabla \varphi_i(\nu_2) -  \nabla \varphi_i(\nu_1): i=1,\dots, k\}$
are linearly independent. This means $d\ge k$. If $d=k$, then \eqref{Mcondition} implies \eqref{hessiannonvanishing}, but otherwise \eqref{Mcondition} may hold without \eqref{hessiannonvanishing} being satisfied.

\

In fact, it is possible to obtain a local version (Theorem \ref{cor2} below) of Theorem \ref{thm},
which holds under a weaker assumption. Let $\mathbf n_1, \dots \mathbf n_{d-k}$\footnote{We consider these vectors as row vectors.} be orthonormal vectors which are perpendicular  to the span of $\{\nabla \varphi_i(\nu_2) -  \nabla \varphi_i(\nu_1): i=1,\dots, k\}$ and set
\[ \mathbf N(\nu_2, \nu_1)=\begin{pmatrix} \mathbf n_1\\ \vdots  \\ \mathbf n_{d-k} \end{pmatrix}.\]
Then  we can  replace the condition \eqref{Mcondition} with
\begin{equation}
\label{normalcondition}
\det  \Big[ \mathbf N(\nu_2, \nu_1)\big( \sum_{i=1}^k t_i H\varphi_i(\nu)\big) \mathbf N^t(\nu_2, \nu_1)\Big]\, \neq 0
\end{equation}
whenever $\nu_1 \in S_1$, $\nu_2 \in S_2$,  $|t|=1$ and $\nu=\nu_1, \nu_2$.  It is  easy to see  that the value of this determinant is independent of the particular choice of  orthonormal vectors $\mathbf n_1,\dots, \mathbf n_{d-k}$, and that the condition \eqref{normalcondition}  is equivalent to \eqref{Mcondition} under the assumption
\eqref{hessianinvertible}.\footnote{Indeed, if $H$, $N$, $D$ are matrices of size $d\times d$, $(d-k)\times d$, $k\times d$, respectively, such that  $ND^t=0$, $\det H\neq 0$, and  rank$\begin{pmatrix}  N^t\, D^t\end{pmatrix}=d$, then    $\det ( NHN^t)\neq 0$ if and only if  $\det (DH^{-1} D^t)\neq 0$ because  $\begin{pmatrix} NH\\ D \end{pmatrix}\begin{pmatrix} N^t\,  D^t \end{pmatrix}=\begin{pmatrix} NHN^t & NHD^t\\ 0 & DD^t \end{pmatrix}$, and  $\begin{pmatrix} N\\ DH^{-1}  \end{pmatrix} \begin{pmatrix} N^t\, D^t \end{pmatrix}=\begin{pmatrix} NN^t & 0\\  DH^{-1}N^{t} & DH^{-1}D^t \end{pmatrix} $.}
If we have \eqref{normalcondition} instead of \eqref{hessiannonvanishing}, then we don't need \eqref{hessianinvertible} to get \eqref{mainbi} for any $\alpha>0$.  More precisely,  we have

\begin{theorem}\label{cor2} Suppose that, for any $\nu_1\in S_1$,  $\nu_2\in S_2$, the vectors $\nabla \varphi_i(\nu_2) -  \nabla \varphi_i(\nu_1),$ $ i=1,\dots, k$, are linearly independent and that \eqref{normalcondition} holds for $\nu_1 \in S_1$, $\nu_2 \in S_2$,  $|t|=1$ and for $\nu=\nu_1, \nu_2$. Then,  for any $\alpha>0$, there is a constant $C_\alpha$ such that
\begin{equation} \label{mainbi}
\| E_1f \, E_2g \|_{L^\frac{d+3k}{d+k}(Q_R)} \le C_\alpha R^\alpha \|f\|_{2}\|g\|_{2},
\end{equation}
where $Q_R$ is a cube of sidelength $R \gg 1$.
\end{theorem}

However, to obtain the global estimates $L^2\times L^2\to L^q$, for $q>\frac{d+3k}{d+k}$, we need to impose a decay condition on the Fourier transform of the surface measure, since it is needed to apply the epsilon removal lemma \cite{BG}.   Under the condition \eqref{hessianinvertible} such a decay estimate follows from the stationary phase method.

For $q\ge 2$, the estimate  \eqref{bilest}  is relatively easier to prove under the  conditions \eqref{hessianinvertible}, \eqref{hessiannonvanishing}. The following may be thought of as a generalization of Theorem 2.3 in \cite{TVV} (see also Theorem 4.2 in \cite{MVV}) which is concerned with elliptic hypersurfaces.  A generalization to general hypersurfaces had already been observed in \cite{L}.   As {a} byproduct this gives estimates for the endpoint cases of $(p,q)$  satisfying  $\frac{1}{p}+\frac{d+3k}{d+k}\frac{1}{2q} = 1$, $q\ge 2$.

\begin{prop} \label{prop:endEst}  Suppose the condition \eqref{Mcondition} holds for  $\nu_1 \in S_1$, $\nu_2 \in S_2$ and  $|t|=1$.   Then, for $q \ge 2$ and $\frac{1}{p}+\frac{d+3k}{d+k}\frac{1}{2q} \le 1$, the estimate \eqref{bilest} holds.
\end{prop}

\begin{remark} \label{uniform} In the proof of the above results we may assume that the aforementioned conditions hold uniformly, by breaking up the extension operator by decomposing $S_1$, $S_2$ into sufficiently small pieces.
That is to say,  there is a constant $c>0$ such that
for $\nu\in S_1\cup S_2 $ and  $|t|= 1$,
\begin{equation}
\label{hessianinvertible2}
\Big|\det \Bpar \sum_{i=1}^k t_i H\varphi_i(\nu)\parB \Big | \ge c
\end{equation}
and,  for $\nu_1, \nu_1' \in S_1$, $\nu_2, \nu_2' \in S_2$,  $|t|\sim 1$ and for $\nu\in S_1\cup S_2 $,
\begin{equation}\label{hessiannonvanishing2}
\Big|\det \Big[ \mathbf D(\nu_1,\nu_2)  \Bpar  \sum_{j=1}^k t_j H\varphi_j(\nu)
 \parB^{-1}  \mathbf D^t(\nu_1', \nu_2')\Big] \Big|\ge c.
\end{equation}
The same holds also for the conditions \eqref{Mcondition} and \eqref{normalcondition}.
\end{remark}

{{}  \subsection*{Necessary conditions for (\ref{bilest})}
By modifying the examples in \cite{TV1} with some specific surfaces} we see that   \eqref{bilest} cannot hold in general, unless
\begin{gather}
q\ge \frac{d+k}{d}, \label{qbb}\\
\frac{1}{p}+\frac{d+3k}{d+k}\frac{1}{2q} \le 1, \label{sqshex}\\
\frac{2(d-k)}{p}+\frac{d+3k}{q} \le 2d. \label{strex}
\end{gather}
{{}  In fact, $(i)$  \eqref{qbb} is necessary for
\eqref{bilest} to hold under \eqref{hessianinvertible}, and  $(ii)$ so is  \eqref{sqshex}  under the assumption that  the matrix $\mathbf D(\nu_1,\nu_2)$ has rank $k$ for $\nu_j\in S_j$, $j=1,2$.  {However}, in general,  \eqref{strex} is not   necessarily required for  \eqref{bilest}, but as is well known there are various $\Phi$ satisfying  \eqref{hessianinvertible} and \eqref{hessiannonvanishing}  for which \eqref{bilest} fails if $ \frac{2(d-k)}{p}+\frac{d+3k}{q} > 2d$.  We show $(i)$ and $(ii)$ in the following paragraphs. }

\smallskip

{{}  {$(i).$} By making use of the stationary phase method together with the condition \eqref{hessianinvertible} it is not difficult to see
that, with suitable choice of $x_0$,  there is a cube $Q$ of sidelength $R \gg 1$ such that
\(
|E_1 (e^{-2\pi i x_0 \cdot \xi}\psi) | \sim |E_2 \psi(x)| \sim R^{-\frac{d}{2}}
\) on $Q$ provided that  supports of $\psi_1,$ $\psi_2$ are small enough. We insert these into \eqref{bilest} to see  $
R^{-\frac{d}{2}}R^{-\frac{d}{2}} R^{\frac{d+k}{q}} \lesssim 1,
$
from which we get \eqref{qbb} by letting $R\to \infty$.\footnote{This can also be shown by making use
of {a} wave packet decomposition (see Lemma \ref{lem:asy}) and randomization.}

\smallskip

{{}  {$(ii)$.} For  $j=1,2$,  let  $\Sigma_j$ be the surface $\{(\xi, \Phi(\xi)): \xi\in S_j\}$, and  denote by $d\sigma_j$ the induced Lebesgue measure on $\Sigma_j$. To see  \eqref{qbb} it is more convenient to consider $f\to  \widehat{ f d\sigma_j}$,  instead of dealing with the operator $E_j$.
 Also, let $\nu_j$ be the center of cube $S_j$ and {let} $\zeta_j=(\nu_j, \Phi(\nu_j))\in\Sigma_j$, $j=1,2$.
 The normal space $\mathbf N_j$ to $\Sigma_j$ at $\zeta_j$ is spanned by
\[  \mathbf n_{j,i} =(-\nabla\varphi_i(\nu_j), e_i), i=1, 2, \dots, k,  \]
where $e_i\in \mathbb R^k$ is the usual unit vector with its $i$-th entry being equal to $1$.
Clearly, these vectors are linearly independent because $\mathbf D(\nu_1,\nu_2)$ has rank $k$. Let $\mathbf p_{n}$, $n=1, \dots, d-k,$ be an orthonormal basis of the orthogonal complement of ${{\rm span}} \{ \mathbf n_{j,i},\,   i=1, 2, \dots, k,\, j=1,2\} $. Let us set, for $j=1,2,$
 \[ \Lambda_j= \{\zeta\in \Sigma_j:   |(\zeta-\zeta_j)\cdot \mathbf n_{3-j,i}|\le \delta,
  |(\zeta-\zeta_j)\cdot \mathbf p_{n}| \le \delta^\frac12,
    i=1, \dots, k, \, n=1, \dots, d-k  \}.\]
Now, we set $f_j=\chi_{\Lambda_j}$, $j=1,2$.  Then  it is easy to see
$| \widehat{ f_j d\sigma_j} (x,t)|\gtrsim \delta^{\frac{d+k}2} $, $j=1,2$, provided that
\[ |(x,t)\cdot \mathbf n_{\ell,i}|\le c\delta^{-1},
  |(x,t)\cdot \mathbf p_{n}| \le c\delta^{-\frac12},
    i=1, \dots, k, \ell=1,2,\, n=1, \dots, d-k\]
with sufficiently small $c>0$. (For example, see the proof Lemma \ref{lem:asy}.) Since \eqref{bilest} implies
$\|  \widehat{ f_1 d\sigma_1}  \widehat{ f_2 d\sigma_2} \|_q\lesssim \|f_1\|_p \|f_2\|_p$, we get
$\delta^{d+k-\frac{d+3k}{2q}}\le C\delta^{\frac{d+k}p}$ and \eqref{sqshex} by letting $\delta\to 0$. }}

\subsection*{Restriction to complex surfaces.} Using the above theorem we can obtain a bilinear restriction estimate for complex quadratic surfaces.
To define the (Fourier) extension operator for a complex surface we first distinguish the dot product and the inner product for complex variables, and define an auxiliary product $\odot$.
For $z, w \in \mathbb C^m$, we define $z \cdot w$, $\langle z, w \rangle$, $z \odot w$  by
\[
z \cdot w = \sum_{j=1}^m z_j w_j,
 \quad \langle z, w \rangle  = \sum_{j=1}^m z_j \bar w_j,
 \quad z \odot w = \mathrm{Re}\, \langle z, w \rangle,
\]
respectively.  {{} Hence, if $z=x+iy$ and $w =u + iv$ for $x, y, u, v \in \mathbb R^m$, {then} $z \odot w = x \cdot u + y \cdot v$. If we identify $\mathbb C^m$ with $\mathbb R^{2m}$ in the usual way, then $z \odot w$ is just the inner product on $\mathbb R^{2m}$.}

Let $n \ge 1$ be an integer and let $D$ be a real symmetric invertible matrix. Then we define
 the complex quadratic surface $\gamma \subset \mathbb C^{n+1}$  by
\begin{equation}  \label{gernPa}
 \gamma(z) = \Big(z,  \frac12\,z^t D z\Big),  \quad z\in \mathbb C^n.
\end{equation}
Now we define the extension operator $E_\gamma f$ by
\begin{equation*} \label{def:T}
E_\gamma f(w) =\int_{\mathbb C^n} e^{2 \pi i  [ w \odot \gamma(z)] } f(z)\,
dz, \quad w\in \mathbb C^{n+1}
\end{equation*}
where we have written $dz$ for $dx \,dy$, $z=x+iy$.
{{}The operator $E_\gamma f$ is an extension operator for surfaces of codimension 2 in $\mathbb R^{2n}$, which is given by
$(x,y,\frac12 \Re (x+iy)^t D (x+iy)), \frac12 \Im (x+iy)^t D (x+iy))$, $x,y\in \mathbb R^n$.}
From Theorem \ref{thm} we can establish the following.
\begin{corollary} \label{prop:Gercomp}
Let $S_1$, $S_2$ be closed cubes in $\mathbb C^n$.
Suppose that,  for any $z_1 \in S_1$ and $z_2 \in S_2$,
\begin{equation} \label{ssc}
|(z_2 - z_1)^t D(z_2 - z_1)| \neq  0.
\end{equation}
Then,  whenever $f$, $g$  are supported on  $S_1$,  $S_2$, respectively,
for $q > \frac{n+3}{n+1}$ and $\frac{1}{p}+ \frac{n+3}{n+1}\frac{1}{2q} <1$, there is a constant $C$ such that
\[
\|  E_\gamma f \, E_\gamma g \|_{L^q(\mathbb C^{n+1})}
\le C \|f\|_{L^p( \mathbb C^n)}\|g\|_{L^p(\mathbb C^{n})}.
\]
\end{corollary}

This theorem can also be stated without using the complex number notation, but the use of the complex number notation makes it easier to derive the linear estimates from the bilinear one.
The condition \eqref{ssc} in $\mathbb C^2$  can be contrasted with that in $\mathbb R^2$.
If $S_1, S_2 \subset \mathbb R^2$ and the eigenvalues of $D$ have the same sign, then the condition \eqref{ssc} is always valid if $\dist(S_1, S_2)\neq 0$.  But, when $S_1, S_2 \subset \mathbb C^2$, the condition \eqref{ssc} may fail even if the separation condition is satisfied.
{{}For instance, if $D$ is the $2 \times 2$ identity matrix, the condition \eqref{ssc} becomes $|(v_1-w_1)^2 + (v_2-w_2)^2| \gtrsim 1$ with $z_1=(v_1,v_2)$ and $z_2=(w_1,w_2)$. Since we may factorize $(v_1-w_1)^2 + (v_2-w_2)^2$ as $ [(v_1-w_1)+i(v_2-w_2)][(v_1-w_1)-i(v_2-w_2)]$, the expression $|(v_1-w_1)^2 + (v_2-w_2)^2|$ may vanish even if $\mathrm{dist}(S_1, S_2) \gtrsim 1$.}  
When $D$ has eigenvalues with different signs, this phenomenon may occur even when $S_1,S_2 \subset \mathbb R^2$; for instance, if
$D$ is the $2\times 2$ diagonal matrix with {diagonal entries $1$ and $-1$}, then we have $x\cdot Dx = x_1^2 - x_2^2 = (x_1+x_2)(x_1-x_2)$.
This {real-variable} case was studied by Lee \cite{L} and Vargas \cite{V}. In the special case that the surface is two-dimensional they could deduce a linear estimate from the bilinear one.

By adapting their argument, we can obtain the following linear estimate.

\begin{theorem} \label{thm:cpl}   Let $n=2$ and $\gamma$ be given by \eqref{gernPa} with a nonsingular real symmetric matrix $D$.
Then, for $q > \frac{10}{3}$ and $\frac{1}{p} + \frac{2}{q} < 1$,
\begin{equation} \label{Lin-Rest}
\|  E_\gamma f \|_{L^q(\mathbb C^3)} \le C \| f \|_{L^p(\mathbb C^2)}
\end{equation}
whenever $f$ is supported in a bounded set.
\end{theorem}

By analogy with the corresponding problem for the paraboloid (elliptic or hyperbolic) in $\bbR^3$, it may be conjectured that \eqref{Lin-Rest} holds if and only if $q>3$ and $\frac{1}{p} + \frac{2}{q} \le 1$. Theorem \ref{thm:cpl} extends the known $(p,q)$ range for the operator $E_\gamma f$ when $D$ is a nonsingular real symmetric matrix.  This result is an analog of the adjoint Fourier restriction estimates for the hyperbolic paraboloid in $\bbR^3$, which is known to hold for the same range of $p$, $q$.
As a special case of the results by Christ (see Lemma 4.3 in \cite{Ch1}) and Mockenhaupt (Theorem 2.11, \cite{Moc}),
it was previously known that $E f$ maps $L^2(\bbR^{4})$ boundedly to $L^4 (\bbR^6)$. Also, the slightly stronger Lorentz space estimate
$\| E f \|_{L^{4,2} (\bbR^6)} \le C \| f\|_{L^2 (\bbR^{4})}$ can be deduced by applying Theorem 1.1 in \cite{BS}.
It is quite likely that the multilinear approach will yield further progress on these problems. We hope to return to this problem in the near future.

\

{\sl Notation.} We adopt the usual convention to let $C$ or $c$ represent strictly positive
constants, whose value may vary from line to line. But these constants will always be independent of $f$, for instance.
We write $A \lc B$ or $B \gc A$ to mean $A \le C B$, and $A\sim B$ means both $A\lc B$ and $B\lc A$.


\section{\texorpdfstring{$L^\frac{4(d+k)}{3d+k}\times L^\frac{4(d+k)}{3d+k}\to L^2$}{L*L -> L2} estimates and proof of Proposition \ref{prop:endEst}}
In this section we show Proposition \ref{prop:endEst}.  Our proof here is different from that in \cite{TVV}. Instead of making use of the boundedness of the
averaging operator, we directly exploit the oscillatory decay estimate which is concealed in the averaging operator.   For this   we need the following  lemma.

\begin{lemma}[\cite{GS}*{Section 1.1}]  \label{lem:SG}
Let $a\in C^\infty_c(\mathbb R^d\times \mathbb R^d\times
\mathbb R^N)$ and set
\[
T_\lambda f(x)
=\int_{\mathbb R^d}\int_{\mathbb R^N}
e^{i\lambda\phi(x,y,\theta)} a(x,y,\theta) \, d\theta\, f(y) \,dy,
\]
where $\phi$ is a smooth function on the support of $a$.  Suppose
$\det\begin{pmatrix}
\phi_{\theta\theta}''& \phi_{x\theta}''\\
\phi_{y\theta}''& \phi_{xy}''
\end{pmatrix} \neq 0$
on the support of $a$ whenever $\phi'_\theta=0$. Then,
$ \| T_\lambda f\|_{2}\lesssim \lambda^{-\frac{d+N}{2}}\|f\|_2.$
\end{lemma}

\begin{proof}[Proof of Proposition \ref{prop:endEst}]
By interpolation with the trivial $L^1\times L^1\to L^\infty$ estimate, it suffices to show
\[ \| E_1f_1 \,E_2f_2 \|_2 \lesssim  \|f_1\|_{\frac{4(d+k)}{3d+k}} \|f_2\|_{\frac{4(d+k)}{3d+k}}.
\]
For fixed $\xi_2,$ set
\[\Phi^{\xi_2}(\xi_1,\eta_1 )= \Phi(\xi_1)  + \Phi(\xi_2) - \Phi(\eta_1)- \Phi(\xi_1+\xi_2-\eta_1)\]
and
\[ I^{\xi_2} (f_1, \bar f_1) =\iint \delta( \Phi^{\xi_2}(\xi_1,\eta_1 ))
f_1(\xi_1)\bar f_1(\eta_1)
d\xi_1 d\eta_1,\]
where $\delta$ is the delta function. Its composition is well defined, since the vectors $\nabla \varphi_i(\nu_2) -  \nabla \varphi_i(\nu_1),$ $ i=1,\dots, k$, are linearly independent.

By Plancherel's theorem
\begin{align*}
\| E_1f_1 \,E_2f_2 \|_2^2 &=
\iiiint \delta( \xi_1 +\xi_2 -\eta_1-\eta_2, \Phi(\xi_1)
+ \Phi(\xi_2) - \Phi(\eta_1)- \Phi(\eta_2))
\\
&\qquad\qquad \times f_1(\xi_1)f_2(\xi_2)\bar f_1(\eta_1)\bar f_2(\eta_2)\, d\xi_1 d\xi_2
d\eta_1 d\eta_2
\\
&= \iiint   \delta( \Phi^{\xi_2}(\xi_1,\eta_1 )) f_1(\xi_1)\bar f_1(\eta_1)f_2(\xi_2)\bar f_2(\xi_1+\xi_2-\eta_1)\,
d\xi_1 d\xi_2 d\eta_1,
\end{align*}
where $f_1$, $f_2$ are assumed to be supported in $S_1,$ $S_2$, respectively.
We claim that
\begin{equation} \label{1-1}
\| E_1f_1 \,E_2f_2 \|_2^2\lesssim  \|f_1\|_{p,1} \|f_2\|_1  \|\bar f_1\|_{p,1} \|\bar f_2\|_\infty,
\end{equation}
where \(p= \frac{d+k}{d}.\) Here $\|f\|_{r,s}$ denotes the norm of Lorentz space $L^{r,s}$. For this we may obviously assume that the functions $f_1,  \bar f_1, f_2,  \bar f_2$ are nonnegative.
In order to show \eqref{1-1} it suffices to prove
\begin{equation}  \label{1-2}
|I^{\xi_2} (f,  g)| \lesssim \|f\|_{p,1}\|\bar g\|_{ p,1}.
\end{equation}

Let $\boldsymbol \psi$ be a smooth function with compact Fourier support contained in $B(0,1)$ such that $\widehat{\boldsymbol \psi}=1$ on $B(0,1/2)$. Since $h(0)=\lim_{j\to \infty} 2^{jk}\int_{\mathbb R^k}  {\boldsymbol \psi}(2^j x)  h(x) dx$ for any Schwartz function $h$, we have
$ \delta= \lim_{j\to \infty} 2^{jk} {\boldsymbol \psi}(2^j x). $
So,  we may  write
\[\delta = \sum_{j=-\infty}^\infty [2^{(j+1)k} {\boldsymbol \psi}(2^{j+1} x) -2^{jk} {\boldsymbol \psi}(2^j x)] =
\sum_{j=-\infty}^\infty 2^{jk} \eta (2^{j} x)\]
where $\eta(x) :=2^k{\boldsymbol \psi}(2x)-{\boldsymbol \psi}(x)$.  By the choice of ${\boldsymbol \psi}$
we see that the Fourier support of $\eta$ is contained in $\{\xi: 1/2<|\xi|\le 2\}$.
We decompose $ I^{\xi_2}(f, g)$ by making use of the above decomposition of $\delta$ to get
\[I^{\xi_2}(f, g)=\sum_{j=-\infty}^\infty I_j(f, g) ,\]
where
\[I_j(f,  g) := 2^{kj}\iint \eta(2^j \Phi^{\xi_2}(\xi_1,\eta_1 ))
f(\xi_1)g(\eta_1)
d\xi_1 d\eta_1.\]
It should be noted that we are assuming that $f,$ $g$ are supported on $S_1$ and $\xi_1+\xi_2-\eta_1\in S_2$.
Using Fourier transform we write $I_j(f_1, \bar f_2)$ as
\[
I_j(f, g)
= 2^{kj} \int \bigg( \iint \widehat \eta(\tau) e^{2^j \tau \cdot \Phi^{\xi_2}(\xi_1,\eta_1 )} d\tau f(\xi_1) d\xi_1
\bigg) g(\eta_1) d\eta_1 .
\]
Now, we will apply Lemma \ref{lem:SG} to the double integral inside the parentheses.
If we set $\phi(\xi_1, \eta_1, \tau)=\tau\cdot \Phi^{\xi_2}(\xi_1,\eta_1 )$, then
\begin{equation*}
\bigg| \det \begin{pmatrix}
\phi_{\tau\tau}''& \phi_{\tau\xi_1}''\\
\phi_{\eta_1\tau}''& \phi_{\xi_1\eta_1}''
\end{pmatrix} \bigg|
= \bigg|  \det
\begin{pmatrix}
0 & \mathbf D(\xi_1, \xi_1+\xi_2-\eta_1)\\
\mathbf D(\eta_1, \xi_1+\xi_2-\eta_1)^t&   \sum_{j=1}^k\tau_jH\varphi_j( \xi_1, \xi_1+\xi_2-\eta_1) \end{pmatrix} \bigg| .
\end{equation*}
So, by the condition \eqref{Mcondition} the last expression does not vanish  since $|\tau|\sim 1$ .
Hence, by Lemma \ref{lem:SG}  it follows that
\[ |I_j(f, g)|\lesssim  2^{-j\frac{d-k}{2}} \|f\|_2\|g \|_2.\]
On the other hand, we have the trivial bound
$ |I_j(f, g)|\lesssim  2^{kj}\|f\|_1\|g \|_1.$
Now we may use a summation method (usually called Bourgain's summation trick) to obtain \eqref{1-2}.

Considering $(f_1,  \bar f_1, f_2,  \bar f_2)\to \| Ef_1 \,Ef_2 \|_2^2$ as a quadrilinear mapping (replacing $\bar f_1$, $\bar f_2$ on the left-hand side by $\bar f_3$ and $\bar f_4$, respectively), we apply M. Christ's multilinear trick \cite{Ch2}.
By symmetry and interpolation we get the estimates
\[  \Big| \iint  Ef_1 \,Ef_2\,\overline{ Ef_3 \,Ef_4 }\,  dxdt \Big|\lesssim \prod_{j=1}^4 \|f_j\|_{p_j,1}\]
for $(1/p_1, 1/p_2, 1/p_3, 1/p_4)$ contained in  the  convex hull of the four points
\[v_1=(1/p, 1/p, 1,0), \, v_2=(1/p,1/p,0,1), \, v_3=(1,0,1/p,1/p), \, v_4=(0,1,1/p,1/p)\]
which is contained in the 3-plane $\Pi=\{u_1+ u_2+u_3+u_4=1+\frac2p \}$. The convex hull has a nonempty interior in $\Pi$, because
$\det(v_1,v_2, v_3, v_4)\neq 0$ as long as $1/p\neq 1/2$. Hence we may apply the multilinear trick to get
\[ \| Ef_1 \, Ef_2 \|_2^2\lesssim  \|f_1\|_{\frac{4(d+k)}{3d+k},4} \|f_2\|_{\frac{4(d+k)}{3d+k},4}  \|\bar f_1\|_{\frac{4(d+k)}{3d+k},4} \|\bar f_2\|_{\frac{4(d+k)}{3d+k},4}. \]
This completes the proof of the proposition.
\end{proof}


\section{ Transversality and the curvature conditions}

In this section we prove  several lemmas that will play crucial roles in proving Theorem \ref{thm}.  These lemmas are related to
the  curvature conditions.

For $R\gg 1$ and $\nu\in S_1\cup S_2$, we set
\[ \pi_{\nu} = \Big\{ (x,t) :  
\big |x + \bpar \sum_{j=1}^k  t_j \nabla \varphi_j (\nu) \parb\big| \le R^{1/2} \Big \}, \quad   R^\delta \pi_{\nu}=\pi_{\nu}+O(R^{\frac12+\delta}).
 \]
Here, for any set $A\subset \mathbb R^{d+k}$ and $\rho>0$, $A+O(\rho)=\{u\in R^{d+k}: \dist(u, A)\le C\rho \}$.
\begin{lemma}\label{intersection of plates}
Suppose that the vectors $\nabla\varphi_j (\nu_2)-\nabla\varphi_j (\nu_1)$, $1 \le j\le k$, are linearly independent  for all $\nu_1\in S_1$ and $\nu_2\in S_2$.   Then, there is a  constant $C$ such that
\[ \pi_{\nu_1}\cap \pi_{\nu_2} \subset B(0, CR^{1/2}).\]
\end{lemma}

\begin{proof} Since the set  $\{\nabla\varphi_j (\nu_2)-\nabla\varphi_j (\nu_1)\}_{j=1}^k$ is linearly independent  for all $\nu_1\in S_1$ and $\nu_2\in S_2$, the map $(t_1,\dots, t_k)\to (t_1,\dots, t_k)^t { \mathbf D}(\nu_1, \nu_2)$ is injective. So,    by continuity and compactness it follows  that there is  a constant $C$ such that, for  all $\nu_1\in S_1$ and $\nu_2\in S_2$,
\[ |(t_1,\dots, t_k)^t { \mathbf D}(\nu_1, \nu_2)|\ge C| (t_1,\dots, t_k)|. \] If $(x,t)\in  \pi_{\nu_1}\cap \pi_{\nu_2}$,  then $\big |x + \bpar \sum_{j=1}^k  t_j \nabla \varphi_j (\nu_i) \parb\big| \le R^{1/2} $ for $i=1,2$. This  gives $|(t_1,\dots, t_k)^t { \mathbf D}(\nu_1, \nu_2)|\le 2 R^{1/2}$.  Hence,  the above inequality yields
$|(t_1,\dots, t_k)|\le CR^{1/2}$. So, we also get  $|x|\le CR^{1/2}$.  This completes the proof.
\end{proof}

As it was already shown in \cites{L, V}, a simple transversality condition between the two wave packets is not enough to obtain a bilinear estimate beyond the range of the linear $L^2\to L^q$ estimate. So, we need to consider the Fourier supports of the wave packets to put a restriction on the permissible wave packets. This makes the geometry of the associated wave packets more favorable.

\smallskip

For given $\nu _1\in S_1$ and $\nu_2' \in S_2$ we define $\Pi_1^{\nu_1,\nu_2'}$ by
\begin{equation}\label{piDef}
\Pi_1^{\nu_1,\nu_2'} = \big\{ \nu_1' \in S_1: \nu_1'+\nu_2'-\nu_1 \in S_2,~ \Phi(\nu_1) + \Phi(\nu_1'+\nu_2'-\nu_1) = \Phi(\nu_1') + \Phi(\nu_2') \big\}.
\end{equation}
Since $\{\nabla\varphi_j (\nu_2)-\nabla\varphi_j (\nu_1)\}_{j=1}^k$ are linearly independent, by the implicit function theorem  we may assume  that $\Pi_1^{\nu_1,\nu_2'}$ is a smooth ($d-k$)-dimensional surface.\footnote{ We may need to assume
that $S_1$ and $S_2$ are small enough.}   We now set
\begin{equation*}
\Gamma_1^{\nu_1,\nu_2'}(R) = \bigcup_{\nu_1' \in \Pi_1^{\nu_1,\nu_2'}}  R^{\delta}\pi_{\nu_1'},
\end{equation*}
which is a $O(R^{\frac12+\delta})$  neighborhood of  the conical set with $k$ null directions.
The transversality between $\Gamma_1^{\nu_1,\nu_2'}$ and the opposite plates $\pi_{\nu_2}$  is important.  Such a transversality is made precise in the following (see Figure 1):

\begin{center}
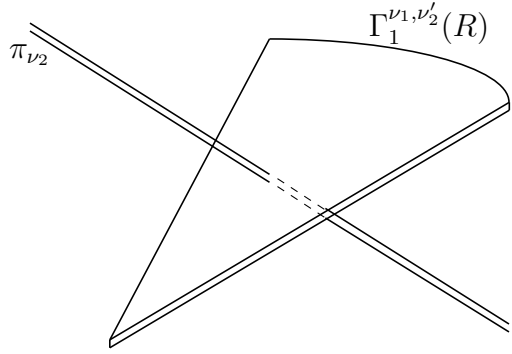
\begin{figure}
\begin{tikzpicture}[scale = 2.1]
	\draw[semithick] (2.5,0) arc (0:90:1.5 and 0.4); 
	\draw[semithick] (2.502,0) -- (0,-1.5);
	\draw[semithick] (2.502,-0.05) -- (0,-1.55);
	\draw[semithick] (2.502,0) -- (2.502,-0.05); 
	\draw[semithick] (0,-1.5) -- (0,-1.55); 
	\draw[semithick] (1,0.4) -- (0,-1.5); 
	\draw[semithick] (-0.5,0.5) -- (1,-0.455);
	\draw[semithick] (-0.5,0.45) -- (1,-0.505);
	\draw[semithick] (1.35,-0.67) -- (2.5,-1.4);
	\draw[semithick] (1.35,-0.73) -- (2.5,-1.45);
    \draw[dashed] (1.05,-0.482) -- (1.35,-0.68); 
    \draw[dashed] (1.05,-0.532) -- (1.35,-0.73);
    \draw (2,0.5) node {$\Gamma_1^{\nu_1,\nu_2'}(R)$};
    \draw (-0.5,0.3) node {$\pi_{\nu_2}$};
\end{tikzpicture}
\caption{Transversality when $k=1$ and $d=2$}
\end{figure}
\end{center}

\begin{lemma}\label{platetrans}
Let  $0<\delta\ll 1$, $u \in \mathbb R^{d+k}$ and set
\[  \widetilde \Gamma_1^{\nu_1,\nu_2'} (R, R^\delta)  =\big\{ (x,t) \in  \Gamma_1^{\nu_1,\nu_2'}(R):   R^{1-\delta}\le   |(x,t)|  \le CR     \big  \}.   \]
Suppose that the conditions \eqref{hessianinvertible} and \eqref{hessiannonvanishing} hold.
Then, if $S_1$ and $S_2$ are sufficiently small,  there exist a constant $C$, independent of $\nu_1,\nu_2'$,  $R$, and a vector $u \in\mathbb R^{d+k}$ such that for some $u'\in\mathbb R^{d+k}$,
\[  \widetilde \Gamma_1^{\nu_1,\nu_2'} (R, R^\delta) \cap \bpar R^\delta \pi_{\nu_2} +u\parb \subset B(u', CR^{\frac12+C\delta}).  \]
\end{lemma}

Note that the set $\widetilde \Gamma_1^{\nu_1,\nu_2'} (R, R^\delta)$ can be represented as a $O(R^{\frac12+\delta})$ neighborhood of a surface.
Let us define the map  $\Phi_1^{\nu_1,\nu_2'}: \Pi_1^{\nu_1,\nu_2'}\times  \mathbb R^{k}\to  \mathbb R^{d+k}  $ by
\begin{align*}
\Phi_1^{\nu_1,\nu_2'} (\nu, t)=\Big(-\sum_{j=1}^k t_j \nabla\varphi_j(\nu), t\Big) .
\end{align*}
Then it is easy to see that
\[  \widetilde \Gamma_1^{\nu_1,\nu_2'} (R, R^\delta)\subset
 \Big\{ \Phi_1^{\nu_1,\nu_2'} (\nu, t):   \nu\in \Pi_1^{\nu_1,\nu_2'},\,\, cR^{1-\delta}\le |t|\le  C R\Big\} + O(R^{\frac12+\delta}).  \]

\begin{proof} After scaling it is sufficient to show that the intersection of the two sets
\[ \Gamma_1=\Big\{ \Phi_1^{\nu_1,\nu_2'} (\nu, t):   \nu\in \Pi_1^{\nu_1,\nu_2'},\,\, R^{-\delta}\le |t|\le C\Big\} +O(R^{-\frac12+\delta}) \]
and
\[ \mathfrak C_2(R^{-\frac12+\delta})= \Big\{ \big(-\sum_{j=1}^k  t_j \nabla \varphi_j (\nu_2),t\big) :  |t|\le  C \Big \} + \widetilde u+ O(R^{-\frac12+\delta}) \]
is contained in a ball of radius $CR^{-\frac12+ C\delta}$. {{}  For $j\ge -C$, let us set
\[ \Gamma_1^j(R^{-\frac12+\delta})=\Big\{ \Phi_1^{\nu_1,\nu_2'} (\nu, t):   \nu\in \Pi_1^{\nu_1,\nu_2'},\,\, 2^{-j-1}\le |t|\le 2^{-j}\Big\} +O(R^{-\frac12+\delta}).\]
Using  homogeneity and a dyadic decomposition in $t$ for $\Gamma_1$, the matter can be reduced to  the case  $2^{-1}\le |t|\le 1$.  That is to say,
\begin{equation}\label{intersection}
\Gamma_1^0(R^{-\frac12+\delta}) \cap \mathfrak C_2(R^{-\frac12+\delta})\subset B(u, C_0 R^{-\frac12+\delta})
\end{equation}
for some $u$ and $C_0>0$. In fact,  using scaling we see \eqref{intersection} implies  that
$\Gamma_1^j (R^{-\frac12+\delta}) \cap \mathfrak C_2(R^{-\frac12+\delta})$ is contained in
 a ball of radius $C_0R^{-\frac12+\delta}$.\footnote{Here we change variables $(x,t) \to 2^{-j}(x,t)$, apply \eqref{intersection}, and reverse the change of variables.}   Since
 $\Gamma_1\subset\cup_{2^{-1}R^{-\delta}\le 2^{j}\le C}\Gamma_1^j$,  $\Gamma_1 \cap \mathfrak C_2(R^{-\frac12+\delta})$ is contained in the union of as many as $\sim\log R$ such balls of radius $C_0R^{-\frac12+\delta}$.  This union of balls is obviously contained in  a ball of radius $CR^{-\frac12+ C\delta}$ since the set  $\Gamma_1 \cap \mathfrak C_2(R^{-\frac12+\delta})$ is connected. }

 Since we may assume that $S_1$ and $S_2$ are sufficiently small, in order to show \eqref{intersection}  it is enough to show that the tangent spaces of  the surfaces
$\Phi_1^{\nu_1,\nu_2'}: \Pi_1^{\nu_1,\nu_2'}\times  \{2^{-1}\le |t|\le 1\}\to \mathbb R^{d+k}$ and
$\big\{ \big(\sum_{j=1}^k  t_j \nabla \varphi_j (\nu_2),t\big) :  |t|\le C  \big \}$ are uniformly transversal to each other.
In fact, since all the underlying sets are compact,  by continuity it is enough to check this  at each point.

Let   $u_0=\Phi_1^{\nu_1,\nu_2'} (\nu_{0}, t_0)$ for $\nu_{0}\in \Pi_1^{\nu_1,\nu_2'}$ and  $2^{-1}\le |t_0| \le 1$.
Let $\mathbf v_{1}, \cdots, \mathbf v_{d-k}$ be orthonormal vectors spanning the tangent space  $ T_{\nu_{0}}\Pi_1^{\nu_1,\nu_2'}$.
Then the tangent space of the parametrized surface $\Phi_1^{\nu_1,\nu_2'}: \Pi_1^{\nu_1,\nu_2'}\times  \{2^{-1}\le |t|\le 1\}\to \mathbb R^{d+k}$ at $u_0$ is spanned by  the vectors
\begin{gather}
\label{tanV1}
(\nabla \varphi_1(\nu_{0}),-1,0,\ldots,0),~ (\nabla \varphi_2(\nu_{0}),0,-1,0,\ldots,0),~ \cdots,~ (\nabla \varphi_k(\nu_{0}),0,\ldots,0,-1)
\intertext{and}
\label{tanV2}
\Big(\mathbf v_i \bpar \sum_{j=1}^{k}  t_{0,j} H \varphi_j(\nu_{0}) \parb,0,\ldots,0    \Big),\quad  i=1,\cdots,d-k.
\end{gather}
On the other hand, the $k$-dimensional plane $\big\{ \big(-\sum_{j=1}^k  t_j \nabla \varphi_j (\nu_2),t\big) :  |t|\le C \big \}$  is spanned by
\begin{equation} \label{lvec}
(\nabla \varphi_1(\nu_2),-1,0,\ldots,0),~ (\nabla \varphi_2(\nu_2),0,-1,0,\ldots,0), \cdots, (\nabla \varphi_k(\nu_2),0,\ldots,0,-1).
\end{equation}
Hence it {suffices} to show that these $d+k$ vectors are linearly independent, or equivalently that the
determinant of the matrix with these vectors as row vectors is nonzero.  After {Gaussian} elimination it is enough to show
\begin{equation}\label{inter}
 \det \begin{pmatrix}
 &   \mathbf V \big(\sum_{j=1}^{k}  t_{0,j} H \varphi_j(\nu_{0})\big)    \\
 &  \mathbf D  (\nu_{0},  \nu_{2})
\end{pmatrix}\neq 0
\end{equation}
where $\mathbf V$ is the  $(d-k)\times d$ matrix having $\mathbf v_{1}, \cdots, \mathbf v_{d-k}$ as its row vectors. Now by \eqref{piDef} we note that the vectors $\mathbf v_{1}, \cdots, \mathbf v_{d-k}$  are orthogonal to the span of the vectors
\[
\nabla \varphi_j(\nu_0+\nu_2'-\nu_1) -\nabla \varphi_j(\nu_0), \quad j=1,\dots, k. \]
Under the assumption that $S_2$ is  small enough, we may replace $ \mathbf D  (  \nu_{0},\nu_{2})$ with
$ \mathbf D  ( \nu_{0},\nu_0+\nu_2'-\nu_1)$. For simplicity we set $\widetilde \nu_2=\nu_0+\nu_2'-\nu_1$.\footnote{We may assume that there is a $c>0$ such that $\big|\det \big[ \mathbf N(\nu_2, \nu_1)\big( \sum_{i=1}^k t_i H\varphi_i(\nu)\big) \mathbf N^t(\nu_2,  \nu_1)\big]\big|\,
	>c$ for $\nu_1\in S_2$ and $\nu_2\in S_2$ (see Remark \ref{uniform}).}   Since $\big(\sum_{j=1}^{k}  t_{0,j} H \varphi_j(\nu_{0})\big)$ is invertible, we need only {show} that
\[\det \mathbf A\neq 0,\] where
\[  \mathbf A=\begin{pmatrix}
&    \mathbf V         \\
& \mathbf D  (\nu_0, \widetilde \nu_2) \big(\sum_{j=1}^{k}  t_{0,j} H \varphi_j(\nu_{0})\big)^{-1}
\end{pmatrix}. \]
{{} Since $\mathbf V \mathbf D^{t} (\nu_0, \widetilde \nu_2)=0$, we  note that
the matrix $\mathbf A \begin{pmatrix} \mathbf V^t & \mathbf D^t (\nu_0, \widetilde \nu_2) \end{pmatrix}$ equals
\[ \begin{pmatrix}
&    I_{d-k}  & 0        \\
& \mathbf D  (\nu_0, \widetilde \nu_2) \big(\sum_{j=1}^{k}  t_{0,j} H \varphi_j(\nu_{0})\big)^{-1}   V^t &    \mathbf D  (\nu_0, \widetilde \nu_2) \big(\sum_{j=1}^{k}  t_{0,j} H \varphi_j(\nu_{0})\big)^{-1}   \mathbf D^t (\nu_0, \widetilde \nu_2)
\end{pmatrix}. \]}
This matrix is clearly invertible thanks to \eqref{hessiannonvanishing}. Hence, so is the matrix $\mathbf A$.  This completes the proof.
\end{proof}


In what follows we show that the following version of Lemma \ref{platetrans} holds, where we assume \eqref{normalcondition} instead of \eqref{hessiannonvanishing}, dropping the condition \eqref{hessianinvertible}.

\begin{lemma} \label{platetrans2} Suppose that, for any $\nu_1\in S_1$,  $\nu_2\in S_2$, $\nabla \varphi_i(\nu_2) -  \nabla \varphi_i(\nu_1),$ $ i=1,\dots, k$ are linearly independent and \eqref{normalcondition} holds for $\nu_1 \in S_1$, $\nu_2 \in S_2$,  $|t|=1$ and for $\nu=\nu_1, \nu_2$.   If $S_1$ and $S_2$ are sufficiently small,  there is a constant $C$, independent of $\nu_1,\nu_2'$,  $R$, and $u$ such that, for some $u'\in\mathbb R^{d+1}$,
\[  \widetilde \Gamma_1^{\nu_1,\nu_2'} (R, R^\delta) \cap \bpar R^\delta \pi_{\nu_2} +u\parb \subset B(u', CR^{\frac12+C\delta}).  \]
\end{lemma}

\begin{proof}
It is sufficient to show that \eqref{inter} holds. As before,  under the assumption that $S_2$ is  small enough, replacing $ \mathbf D  (  \nu_{0},\nu_{2})$ with
$ \mathbf D  ( \nu_{0},\widetilde \nu_2)$,  $\widetilde \nu_2=\nu_0+\nu_2'-\nu_1$.  We need only {show} that
\begin{equation*}
\det \begin{pmatrix}
& \mathbf V \big(\sum_{j=1}^{k}  t_{0,j} H\varphi_j(\nu_{0})\big) \\
&\mathbf D(\nu_{0},\widetilde\nu_{2})
\end{pmatrix}\neq 0.
\end{equation*}
Since vectors $\mathbf v_{1}, \cdots, \mathbf v_{d-k}$  are orthogonal to the row vectors of $\mathbf D(\nu_0, \widetilde \nu_2)$,
by multiplying  the nonsingular matrix  $(\mathbf V^t, \mathbf D^t(\nu_0, \widetilde \nu_2))$  to the matrix inside {the} determinant from the right, we see that the above is equivalent to
\begin{equation*}
\det \begin{pmatrix}
& \mathbf V \big(\sum_{j=1}^{k}  t_{0,j} H \varphi_j(\nu_{0})\big) \mathbf V^t  &  \mathbf V \big(\sum_{j=1}^{k}  t_{0,j} H \varphi_j(\nu_{0})\big)  \mathbf D^t  (\nu_{0}, \widetilde \nu_{2})    \\
&0    &  \mathbf D  (\nu_{0}, \widetilde \nu_{2}) \mathbf D^t  (\nu_{0}, \widetilde \nu_{2})
\end{pmatrix}\neq 0.
\end{equation*}
Since the matrix $\mathbf D(\nu_{0}, \widetilde \nu_{2}) \mathbf D^t  (\nu_{0}, \widetilde\nu_{2})$ is nonsingular, it is clear that the above is equivalent to $\det[\mathbf V \big(\sum_{j=1}^{k}  t_{0,j} H \varphi_j(\nu_{0})\big) \mathbf V^t]\neq 0 $, which is  the condition \eqref{normalcondition}.
\end{proof}


\section{Proof of Theorem \ref{thm}}

In this section we will prove Theorem \ref{thm}. Our proof is similar to that in \cite{L} (also see \cite{T}).
To prove Theorem \ref{thm}, we need only show that, for $p>\frac{d+3k}{d+k}$,
\begin{equation*}
\| E_1f \, E_2g \|_{p} \le C \|f\|_{2}\|g\|_{2}
\end{equation*}
since we can obtain the desired conclusion by interpolating this estimate with the trivial estimate
$
\| E_1f \, E_2g \|_{\infty} \le \|f\|_{1}\|g\|_1.
$
By an $\epsilon$-removal argument  \cites{TV1, BG},  it is sufficient to show that  \eqref{mainbi} holds for any $\alpha>0$.
In fact, by the assumption  that  $\sum_{ j=1}^{k} t_{j} H\varphi_{j} (\nu)$ is nonsingular for $\nu\in \supp f\cup \supp g$ as long as $|t|=1$, it follows that
\[|E_\kappa (a_\kappa )(x,t)| \lesssim (|x|+|t|)^{-\frac{d}{2}}, \quad \kappa=1, 2,\]
where $a_1,$ $a_2$ are smooth bump functions which vanish on the supports of $f$ and $g$, respectively.  This can be shown by the stationary phase method.
Hence,  the arguments in \cites{TV1, BG} work here without modification.

\begin{prop}\label{prop}  Let $0 < \delta \ll 1$.
If \eqref{mainbi} holds, then  for any $\epsilon>0$
\begin{equation} \label{mainbi2}
\| E_1f \, E_2g \|_{L^\frac{d+3k}{d+k}(Q_R)} \le C_{\epsilon}  R^{\max (\alpha(1-\delta), C\delta)+\epsilon} \|f\|_{2}\|g\|_{2}
\end{equation}
with $C$ {independent} of $\delta$.
\end{prop}
\noindent By iterating finitely many times  the implication in  Proposition \ref{prop}, we can easily obtain the
estimate \eqref{mainbi} for any $\alpha>0$.

\subsection{Wave packet decomposition}
In this section we decompose the function $Ef$ into wave packets.
Let $R \gg 1$. We denote by \[\mathcal L=\mathcal L(R) := R^{1/2} \mathbb Z^{d}, \quad \mathcal V =\mathcal V(R) := R^{-1/2} \mathbb Z^d.\] Let $\psi$ be a nonnegative Schwartz function such that
$\widehat \psi$ is supported on $B(0,1)$ and $\sum_{k \in \mathbb Z^d} \psi(\cdot - k) =1$. Also, let
$\zeta$ be a smooth function supported on $B(0,1)$ and $\sum_{k \in \mathbb Z^d} \zeta(\cdot - k) =1.$

For $\ell\in \mathcal L$, $\nu\in \mathcal V$ we set $ \psi_{\ell}(x) := \psi(\frac{x-\ell}{R^{1/2}}),$ $\zeta_{\nu}(\xi)=:\zeta(R^{1/2}(\xi- \nu)), $
and for a given function $f$, we  define $f_{\ell, \nu}$ by
\[
f_{\ell, \nu}= \mathcal F\big( \psi_{\ell} \mathcal F^{-1} (\zeta_{\nu} f)\big),
\]
where $\mathcal F$, $\mathcal F^{-1}$ denote Fourier transform and the inverse Fourier transform, respectively.  Then, it follows that
$
f = \sum_{\nu \in \mathcal V} \sum_{\ell \in \mathcal L} f_{\ell, \nu}.
$
Hence we may write
\begin{equation} \label{eqn:WPD}
Ef = \sum_{\nu \in \mathcal V} \sum_{\ell \in \mathcal L}
E f_{\ell,\nu}.
\end{equation}

\begin{lemma} \label{lem:asy}
If $|t| \lesssim  R $, then
\begin{equation}\label{wavepacketdecay}
|Ef_{\ell,\nu}(x,t)| \le C_{N} \, \Big( 1 + R^{-\frac12}\big|x-\ell + \sum_{j=1}^k t_j \nabla\varphi_j(\nu)\big|\Big)^{-N}M (\mathcal F^{-1} (\zeta_{\nu} f))(\ell)
\end{equation}
for all $N \ge 0$. Here, $Mf$ is the Hardy-Littlewood maximal function of $f$.
\end{lemma}

\begin{proof}
Since $f_{\ell,\nu}$ is supported in $B(\nu, 3R^{-1/2})$, 
multiplying by a harmless smooth bump function $\widetilde \chi$ supported in $B(0,5)$ and satisfying $\widetilde \chi=1$ on $B(0,3)$,
we may write
\[
E f_{\ell,\nu}(x,t) = \int K(x - z,t) \psi_{\ell}(z) \mathcal F^{-1} f_\nu(z) dz,
\]
where $
K(x,t) = \int e^{2 \pi i (x \cdot \xi + t \cdot \Phi(\xi))} \chi\big(R^{1/2}(\xi-\nu)\big) d\xi.
$
Changing variables $\xi\to  R^{-1/2}\xi +\nu$,
\[
K(x,t) = R^{-d/2} e^{2\pi i x \cdot \nu} \int e^{2 \pi i (R^{-1/2} x \cdot \xi  + t \cdot \Phi(R^{-1/2} \xi + \nu))} \chi(\xi) d \xi.
\]
Since $|t|\lesssim R$,  $
\nabla_\xi (R^{-1/2} x \cdot \xi + t \cdot \Phi(R^{-1/2} \xi + \nu)) = R^{-1/2} \big (x +  \sum_{j=1}^k t_j \nabla\varphi_j(\nu)\big)+O(1).
$ This follows by {Taylor's} expansion.  Hence, by repeated integration by parts we get
\[
| K(x,t) | \le C_N R^{-d/2} \Bpar 1 + R^{-1/2} \big| x+ \sum_{j=1}^k t_j \nabla\varphi_j(\nu) \big| \parB^{-N}.
\]
Once this is established, \eqref{wavepacketdecay} follows by a standard argument. See \cite{L} for the details.
\end{proof}

From the above lemma we see that $Ef_{\ell,\nu}$ is essentially supported on
\begin{equation} \label{def:tube}
\pi_{\ell,\nu} =   \pi_{\nu} +(\ell, 0) .
\end{equation}
If $\pi=\pi_{\ell,\nu}$, we define $\nu(\pi)=\nu$, which may be considered as  the (generalized) direction of $\pi$.

\medskip

The following is the main lemma of this section.

\begin{lemma} \label{packetdecomp} Let $R\gg 1$. Then,  $Ef$ can be rewritten as
\begin{equation} \label{eqn:WPD'}
Ef(x,t) = \sum_{ (\ell, \nu) \in \mathcal L\times \mathcal V} c_{\ell, \nu}  P_{\ell, \nu}(x,t)
\end{equation}
and  $c_{\ell, \nu} $, $P_{\ell, \nu}$ satisfy the following:

\begin{enumerate}
\item[$(i)$]
 $\mathcal F (P_{{\ell,\nu}}(\cdot, t))$ is supported in the disc $D(\nu, CR^{-1/2})$.

\smallskip

\item[$(ii)$]
If $|t| \lesssim R$, then for any $N \ge 0$
\begin{equation*}
| {P_{\ell,\nu}}(x,t)| \le C_N R^{-d/4} \Bpar 1 + R^{-1/2}\big| x-\ell + \bpar \sum_{j=1}^k t_j \nabla \varphi_j (\nu) \parb\big|\parB^{-N}.
\end{equation*}
\item[$(iii)$]  $\big( \sum_{(\ell,\nu )\in \mathcal L\times \mathcal V} |c_{\ell, \nu}|^2 \big)^{1/2} \lesssim
\|f\|_2.$

\smallskip

\item[$(iv)$]  If $|t| \lesssim R$, then  $\big\| \sum_{(\ell,\nu )\in \mathcal W} P_{\ell,\nu} (\cdot, t) \big\|_2^2 \lesssim \#\mathcal W$ for any $\mathcal W\subset \mathcal L\times \mathcal V$.
\end{enumerate}
\end{lemma}

\begin{proof}
We define $c_{{\ell,\nu}}$ and $P_{\ell,\nu} $ by
\[
 c_{{\ell,\nu}} = R^{d/4} M(\mathcal F^{-1} f_\nu)(\ell),   \quad
P_{{\ell,\nu}} (x,t)= c_{\ell,\nu}^{-1}  Ef_{\ell,\nu}(x,t)
\]
where $M$ denotes the Hardy-Littlewood maximal {function}.
Then we have \eqref{eqn:WPD'} from \eqref{eqn:WPD}.
 Since
$
Ef_{\ell,\nu}(\cdot,y)  = \mathcal F^{-1}(e^{2 \pi i y \Phi} f_{\ell,\nu})
$, $Ef_{\ell,\nu}(\cdot,y) $ has a Fourier support contained in $\supp f_{\ell,\nu}$, which is in turn contained in $D(\nu,CR^{-1/2})$. Thus $(i)$ follows and so does $(ii)$ from Lemma \ref{lem:asy}.

In order to show $(iii)$, note that
\begin{equation} \label{eqn:csum}
\sum_{(\ell,\nu )\in \mathcal L\times \mathcal V} |c_T|^2 =  R^{d/2} \sum_{(\ell,\nu)\in  \mathcal L\times \mathcal V}  M(\mathcal F^{-1} (\zeta_{\nu} f))(\ell) ^2.
\end{equation}
Since $\zeta_{\nu} f$ is supported on $B(\nu, CR^{1/2})$, $M(\mathcal F^{-1}  (\zeta_{\nu} f))(x) \sim M(\mathcal F^{-1} (\zeta_{\nu} f))(x')$ if $|x-x'| \lesssim R^{1/2}$. Hence, from the Hardy-Littlewood maximal theorem and the Plancherel theorem  we have that, for each $\nu$,
\begin{equation*}
R^{d/2} \sum_{(\ell,\nu) \in \mathcal L\times \mathcal V} |M(\mathcal F^{-1} (\zeta_{\nu} f))(\ell)|^2
\lesssim \int |M(\mathcal F^{-1}  (\zeta_{\nu} f))(x)|^2 dx \lesssim \|\zeta_{\nu} f\|_2^2.
\end{equation*}
 Combining this and \eqref{eqn:csum} we obtain
$
\sum_{{(\ell,\nu)} \in \mathcal L\times \mathcal V} |c_{{\ell,\nu}}|^2 \lesssim \sum_{\nu \in \mathcal V} \|\zeta_{\nu} f\|_2^2 \lesssim \|f\|_2^2,
$
and $(iii)$.

Finally, we consider  $(iv)$.  Since $\sum_{\ell: (\ell, \nu)\in \mathcal W } P_{\ell,\nu} (\cdot,t)$ is Fourier-supported in $D(\nu,CR^{-1/2})$, which have bounded overlap as $\nu$ varies over $\mathcal V$. By Plancherel's theorem,
\[
\Big\| \sum_{{(\ell, \nu)\in \mathcal W }}  P_{\ell,\nu} (\cdot, t) \Big\|_2^2 \lesssim \sum_{\nu \in \mathcal V} \Big\| \sum_{{\ell: (\ell, \nu)\in \mathcal W }} P_{{\ell, \nu}}(\cdot, t) \Big\|_2^2.
\]
From $(ii)$  it is easy to see that $\Big\| \sum_{{\ell: (\ell, \nu)\in \mathcal W }} P_{{\ell, \nu}}(\cdot, t) \Big\|_2^2\lesssim
\#\{\ell: (\ell, \nu)\in \mathcal W \} $. Hence, combining this with the above gives $(iv)$.
\end{proof}


\subsection{Dyadic pigeonholing and reduction}
From now on we will prove Proposition \ref{prop}. For simplicity we set
\[ p_0=\frac{d+3k}{d+k}. \]

By translation invariance we may assume that  $Q_R$ is centered at the origin.
Let
\[  \mathcal W_i\subset\{(\ell, \nu)\in \mathcal L\times \mathcal V:  \nu\in S_i +O(R^{-\frac12}) \},\,\, i=1,2.  \]
By Lemma \ref{packetdecomp} and the standard reduction with pigeonholing, which may only cause a loss of  $(\log R)^C$ (see \cites{L, T}), the matter is reduced to showing
\[
  \Big\| \sum_{\omega_1 \in  \mathcal W_1}  P_{\omega_1}
  \sum_{\omega_2  \in \mathcal W_2}  P_{\omega_2}
  \Big\|_{L^{p_0}(Q_R)} \lessapprox (R^{(1-\delta)\alpha} +
  R^{C\delta}) (\#\mathcal W_1\#\mathcal W_2)^\frac12,
\]
whenever $P_{\omega_1}$, $P_{\omega_2}$ {satisfy} $(i),(ii), (iv)$ in Lemma \ref{packetdecomp}.
Here $A\lessapprox B$ means $A\le C_\epsilon R^\epsilon B$ for any $\epsilon>0$.

By a further pigeonholing  argument we specify {the} associated quantities in dyadic scales.
Let $\mathcal  Q$ be {a} collection of almost disjoint cubes of the same sidelength $\sim R^{1/2}$, which cover $Q_R$.
For each $q \in \mathcal Q$ we define
\[
  \mathcal W_j(q) =\{ \omega_j \in \mathcal W_j: \pi_{\omega_j}\cap R^{\delta}q \neq
  \emptyset \}.
\]
For dyadic numbers $\rho_1, $ $\rho_2$  with $1 \le \rho_1, \rho_2 \le R^{100d}$,
we define
\begin{equation} \label{defQ}
  \mathcal Q(\rho_1, \rho_2) = \{q \in\mathcal Q: \rho_j \le \#\mathcal W_j(q) < 2 \rho_j,\ \ j=1,2\}.
\end{equation}
For  $\omega \in \mathcal W_1\cup \mathcal W_2$, we set
\begin{equation*}
  \lambda(\omega; \rho_1, \rho_2) = \#\{ q \in \mathcal Q(\rho_1,
  \rho_2): \pi_{\omega} \cap R^{\delta}q \neq \emptyset\}.
\end{equation*}
For dyadic numbers $1 \le \lambda \le
R^{100d}$ we define
\begin{equation} \label{Treduc}
  \mathcal W_j[\lambda;\rho_1, \rho_2] = \{ \omega_j \in \mathcal W_j:
  \lambda \le \lambda({\omega_j}; \rho_1, \rho_2) < 2\lambda \}, \,\, j=1,2.
\end{equation}
By a standard pigeonhole argument, it is sufficient to show
\begin{equation}
\label{redMa2}
\begin{aligned}
 & \Big(\sum_{q \in \mathcal Q(\rho_1, \rho_2)}
  \Big\| \sum_{\omega_1 \in  \mathcal W_1[\lambda_1;\rho_1, \rho_2]} P_{\omega_1}
  \sum_{\omega_2 \in  \mathcal W_2[\lambda_2;\rho_1, \rho_2]} P_{\omega_2}
  \Big\|^{p_0}_{L^{p_0}(q)} \Big)^{1/p_0} \\
&\qquad \qquad\lessapprox (R^{(1-\delta)\alpha} + R^{C\delta}) \,
 ( \#\mathcal W_1 \, \#\mathcal W_2)^{1/2}. \end{aligned}
\end{equation}

For the rest of the proof we assume that   $q\in \mathcal Q(\rho_1, \rho_2)$, $\omega_1 \in  \mathcal W_1[\lambda_1;\rho_1, \rho_2]$ and $\omega_2 \in  \mathcal W_2[\lambda_1;\rho_1, \rho_2]$ if it is not mentioned otherwise. So, the above sums are denoted simply by
$\sum_{q},$ $\sum_{\omega_1},$ and $\sum_{\omega_2}$, respectively.

\subsection{Induction argument}
For brevity let us put
\[
 \Delta = \bigcup_{q\in \mathcal Q(\rho_1, \rho_2)} q.  \]
Let $\{B\}$ be {a} collection of almost disjoint cubes of the same sidelength $R^{1-\delta}$, which cover $Q_R$. Then
\begin{equation}\label{redMa22}
 \text{  LHS of  \eqref{redMa2}} \le \sum_{B}
  \Big\| \sum_{\omega_1} P_{\omega_1}
  \sum_{\omega_2 } P_{\omega_2}
  \Big\|_{L^{p_0}(\Delta  \cap B)}.
\end{equation}
We define a relation $\sim$ between $\omega_1$ (or $\omega_2$) and the cubes in $\{B\}$. For each $\omega \in \mathcal W_1[\lambda_1;\rho_1,
\rho_2]\cup \mathcal W_2[\lambda_2;\rho_1,
\rho_2]$, we define $B^*(\omega) \in \{B\}$ to be the cube which maximizes the quantity
\begin{equation} \label{B*max}
  \#\{ q \in \mathcal Q(\rho_1, \rho_2):  \pi_\omega\cap R^\delta q \neq
  \emptyset; \ q \cap B \neq \emptyset \}.
\end{equation}
Then the relation $\sim$ is defined as follows:
\[  \omega\sim B \text{ if  }  B\cap 10 B^*(\omega) \neq \emptyset.  \]
Here $10 B^*(\omega)$ is the cube which has the same center as $B^*(\omega)$ and   sidelength  10 times as large as that of $B^*(\omega)$.
Using this relation we divide the sum into three parts to get
\begin{align}
&\sum_{B} \Big\| \sum_{\omega_1 } P_{\omega_1}
  \sum_{\omega_2 } P_{\omega_2}
  \Big\|_{L^{p_0}(\Delta  \cap B)} \nonumber
\le
 \sum_{B} \Big\| \sum_{\omega_1 : \omega_1 \sim B} P_{\omega_1}
  \sum_{\omega_2: \omega_2 \sim B} P_{\omega_2}
  \Big\|_{L^{p_0}(\Delta  \cap B)}\nonumber \\
+&
  \sum_{B}\Big\| \sum_{\omega_1 :\omega_1 \sim B } P_{\omega_1}
  \sum_{\omega_2: \omega_2 \nsim B} P_{\omega_2}
  \Big\|_{L^{p_0}(\Delta  \cap B)}+
  \sum_{B}\Big\| \sum_{\omega_1 :\omega_1 \nsim B} P_{\omega_1}
  \sum_{\omega_2} P_{\omega_2}
  \Big\|_{L^{p_0}(\Delta  \cap B)}\nonumber.
\end{align}

We will first show that
\begin{equation} \label{glPEst}
  \sum_{B}
  \Big\| \sum_{\omega_1 : \omega_1 \sim B} P_{\omega_1}
  \sum_{\omega_2: \omega_2 \sim B} P_{\omega_2}
  \Big\|_{L^{p_0} (\Delta  \cap B)}
  \lesssim  R^{(1-\delta)\alpha} (\#\mathcal W_1\#\mathcal W_2)^{1/2}.
\end{equation}
By applying the hypothesis \eqref{mainbi}, $(iv)$ in Lemma \ref{packetdecomp},  and the Cauchy-Schwarz inequality,
\begin{align*}
  \sum_{B}
  \Big\| &\sum_{\omega_1 : \omega_1 \sim B} P_{\omega_1}
  \sum_{\omega_2: \omega_2 \sim B} P_{\omega_2}
  \Big\|_{L^{p_0} (\Delta  \cap B)}  \le C
R^{(1-\delta)\alpha}
 \prod_{j=1}^2 \Big(\sum_{B}\#\{\omega_j : \omega_j \sim
  B\}\Big)^{1/2}\,.
\end{align*}
From the definition of the relation $\sim$ it is clear that $\#\{B : \omega_j \sim  B\} \le C$. Hence, for $j=1,2$
\begin{equation*}
  \sum_{B}\#\{\omega_j : \omega_j \sim  B\} = \sum_{\omega_j} \#\{B :
    \omega_j \sim  B\} \lesssim  W_j.
\end{equation*}
By inserting this into the previous inequality, we get \eqref{glPEst}.

\

Now, to prove \eqref{redMa2} it is enough to show
\[
  \Big\| \sum_{\omega_1 :\omega_1 \sim B } P_{\omega_1}
  \sum_{\omega_2: \omega_2 \nsim B} P_{\omega_2}
  \Big\|_{L^{p_0} (\Delta  \cap B)} \lessapprox R^{C\delta}  (\#\mathcal W_1 \, \#\mathcal W_2)^{1/2}
\]
and
\begin{equation} \label{MainP}
  \Big\| \sum_{\omega_1 :\omega_1 \nsim B} P_{\omega_1}
  \sum_{\omega_2} P_{\omega_2}
  \Big\|_{L^{p_0} (\Delta  \cap B)} \lessapprox R^{C\delta}  (\#\mathcal W_1 \, \#\mathcal W_2)^{1/2}.
\end{equation}
The proofs of these two estimates are similar. So, we will only prove
\eqref{MainP}.
By Plancherel's theorem, $\|E f(\cdot, t)\|_2\le  \|f\|_2$  for all $t\in \mathbb R^k$. Integration in $t$ gives
$\|E f\|_{L^2(Q_R)}\lesssim R^\frac k2 \|f\|_2$. By the Schwarz inequality it follows that
\[ \|E_1 f E_2 g\|_{L^1(Q_R)}\lesssim   R^k  \|f\|_2\|g\|_2. \]
Combining this with $(iv)$ in Lemma \ref{packetdecomp} yields
\begin{equation} \label{L1}
  \Big\| \sum_{\omega_1 :\omega_1 \nsim B} P_{\omega_1}
  \sum_{\omega_2} P_{\omega_2}
  \Big\|_{L^{1} (\Delta  \cap B)} \lesssim R^k  (\#\mathcal W_1\#\mathcal W_2)^{1/2}.
\end{equation}
Hence, the \eqref{MainP} follows from interpolation between \eqref{L1} and
\begin{equation} \label{L2}
  \Big\| \sum_{\omega_1 :\omega_1 \nsim B} P_{\omega_1}
  \sum_{\omega_2} P_{\omega_2}
  \Big\|_{L^{2} (\Delta  \cap B)} \lessapprox R^{C\delta}R^{-\frac{d-k}{4}}  (\#\mathcal W_1\#\mathcal W_2)^{1/2}.
\end{equation}
Now it remains to show the $L^2$-estimate \eqref{L2}.

\subsection{\texorpdfstring{$L^2$}{L2} estimate } To prove \eqref{L2} it suffices to show
\begin{equation} \label{L2Sec}
  \sum_{q \in \mathcal Q(\rho_1,\rho_2): q \subset 2B}
  \Big\| \sum_{\omega_1 :\omega_1 \nsim B} P_{\omega_1}
  \sum_{\omega_2} P_{\omega_2}
  \Big\|^2_{L^{2}(q)} \lessapprox R^{C\delta}R^{-(d-k)/2} \#\mathcal W_1 \#\mathcal W_2.
\end{equation}
For $j=1,2,$ let us set
\begin{gather*}
  \mathcal W_j(q) = \{ \omega_j \in \mathcal W_i[\lambda_j;\rho_1, \rho_2]:
  \omega_j \cap R^\delta q \neq  \emptyset \}, \quad
  \mathcal W_j^{\nsim B}(q) = \{ \omega_j \in \mathcal W_j(q): \omega_j \nsim
  B \}.
\end{gather*}
Then by $(ii)$ in Lemma \ref{packetdecomp} we may discard some harmless terms, whose contributions are $O(R^{-C\delta})$. Hence,  it suffices to show
\begin{equation} \label{L2Thi}
  \sum_{q \in \mathcal Q(\rho_1,\rho_2): q \subset 2B}
  \Big\| \sum_{\omega_1 \in  \mathcal W_1^{ \nsim B}(q)} P_{\omega_1}
  \sum_{\omega_2(q)} P_{\omega_2}
  \Big\|^2_2 \lessapprox R^{C\delta}R^{-(d-k)/2} \#\mathcal W_1 \,\#\mathcal W_2.
\end{equation}

By using Plancherel's theorem we write
\begin{align*} \label{innForm}
&\Big\| \sum_{\omega_1 \in  \mathcal W_1^{ \nsim B}(q)} P_{\omega_1}
\sum_{\omega_2\in \mathcal W_2(q)} P_{\omega_2}
\Big\|^2_2 \\
&\qquad \qquad =\sum_{\omega_1 \in  \mathcal W_1^{ \nsim B}(q)}
\sum_{\omega_2' \in \mathcal W_2(q)}
\sum_{\omega_1' \in \mathcal W_1^{ \nsim B}(q)}
\sum_{\omega_2\in \mathcal W_2(q)}
\Big\langle  \widehat P_{\omega_1} \ast \widehat P_{\omega_2},
\widehat  P_{\omega_1'} \ast \widehat P_{\omega_2'} \Big\rangle.
\end{align*}
Let us write
$\omega_j=(\ell_j, \nu_j),$ $\omega_j'=(\ell_j', \nu_j'),$ $j=1,2.$
For any $\nu_1 \in S_1$, $\nu_2' \in S_2$, we define $\mathcal W_1^{\nsim B}(q;\nu_1,\nu_2')$ by
\[\mathcal W_1^{\nsim B}(q;\nu_1,\nu_2')=\Big\{ \omega_1'=(\ell_1',\nu_1') \in \mathcal W_1^{ \nsim B}(q): \nu_1'\in \Pi_1^{\nu_1,\nu_2'}  +O(R^{-1/2})\Big\}.\]
Then $\widehat P_{\omega_1} \ast \widehat P_{\omega_2}$ is supported on the
$O(R^{-1/2})$-neighborhood of the point $(\nu_1+\nu_2,
\Phi(\nu_1)+\Phi(\nu_2))$. So the inner product $\Big\langle  \widehat
P_{\omega_1} \ast \widehat P_{\omega_2},
  \widehat P_{\omega_1'} \ast \widehat P_{\omega_2'} \Big\rangle$ vanishes
  unless
\begin{align*}
  \nu_1 + \nu_2 &= \nu_1' + \nu_2' + O(R^{-1/2}), \quad
  \Phi(\nu_1) + \Phi(\nu_2) = \Phi(\nu_1') + \Phi(\nu_2') +O(R^{-1/2}).
\end{align*}
Thus, for given $\nu_1$ and $\nu_2'$, we see that $\nu'_1$ is contained in a $O(R^{-1/2})$-neighborhood of $\Pi_1^{\nu_1,\nu_2'}$, which is defined by \eqref{piDef}. Once $\nu_1$, $\nu'_1$ and $\nu_2'$ are given, then there are only $O(1)$ many $\nu_2$, since $\nu_2$ should be in a $O(R^{-1/2})$-neighborhood of the point $\nu_1+ \nu_1'-\nu_2'$.  Therefore,
\[
\Big\| \sum_{\omega_1 \in  \mathcal W_1^{ \nsim B}(q)} P_{\omega_1}
  \sum_{\omega_2\in \mathcal W_2(q)} P_{\omega_2}
  \Big\|^2_2
  \lesssim R^{-(d-k)/2}
  \sum_{\omega_1 \in  \mathcal W_1^{ \nsim B}(q)}
  \sum_{\omega_2' \in \mathcal W_2(q)}
    \#\mathcal W_1^{\nsim B}(q;\nu_1,\nu_2'),
\]
where we also used 
\[
  \Big|\Big\langle  P_{\omega_1} \, P_{\omega_2},
  P_{\omega_1'} \, P_{\omega_2'} \Big\rangle \Big| \lesssim
 R^{-(d-k)/2}.
 \]
 This follows from  $(ii)$ in Lemma \ref{packetdecomp} and the transversality between $\pi_{\omega_1}$  ($\pi_{\omega_1'}$) and  $\pi_{\omega_2}$ {($\pi_{\omega_2'}$), respectively}. %
Hence,  \eqref{L2Thi} follows if we show
\begin{equation} \label{L2Four}
 \max_{q\subset 2B,\nu_1,\nu_2'} \#\mathcal W_1^{\nsim B}(q; \nu_1, \nu_2')  \sum_{q \in \mathcal Q(\rho_1,\rho_2): q \subset 2B}
  \#\mathcal W_1^{ \nsim B}(q)
  \#\mathcal W_2(q)
  \lesssim   R^{C\delta}\#\mathcal W_1\, \#\mathcal W_2.
\end{equation}

We will prove \eqref{L2Four}, assuming for the moment  that
\begin{equation} \label{finalE}
  \max_{q\subset 2B,\nu_1,\nu_2'} \#\mathcal W_1^{ \nsim B}(q; \nu_1, \nu_2') \lesssim R^{C\delta}
  \frac{\#\mathcal W_2}{\lambda_1
  \rho_2}.
\end{equation}
To this end it is enough to show
\[
\sum_{q \in \mathcal Q(\rho_1,\rho_2): q \subset 2B}
\#\mathcal W_1^{ \nsim B}(q)
\#\mathcal W_2(q) \lesssim \lambda_1 \rho_2 \#\mathcal W_1.
\]
Recalling $\#\mathcal W_2(q)\lesssim \rho_2$, we see that the left hand side is bounded by
\[
C\rho_2 \sum_{q \in \mathcal Q(\rho_1,\rho_2)} \#\mathcal W_1(q).
\]
Changing the order of summation, we see this in turn is bounded by $C\rho_2 \sum_{\omega_1} \#\{ q \in \mathcal Q(\rho_1,\rho_2): \pi_{w_1}\cap R^\delta q \}$.  Since  $\#\{ q \in \mathcal Q(\rho_1,\rho_2): \pi_{w_1}\cap R^\delta q \}\lesssim \lambda_1$, the desired inequality \eqref{L2Four} follows.

\subsection{Proof of (\ref{finalE})}
Fix $q \subset 2B$, $\nu_1 \in S_1$ and $\nu_2' \in S_2$. Let us
consider the set
\begin{align*}
  \mathbf S:=\Big\{(\widetilde q, \omega_1, \omega_2) \in \mathcal Q&( \rho_1, \rho_2) \times
  \mathcal W_1^{\nsim B}(q; \nu_1, \nu_2') \times \mathcal W_2 :
\\
 &\pi_{\omega_1} \cap R^{\delta}\widetilde q \neq
  \emptyset,\, \pi_{\omega_2} \cap R^{\delta}\widetilde q \neq
  \emptyset,\ \ \mathrm{dist}(\widetilde q, q) \ge R^{1-\delta} \Big \}.
\end{align*}
To prove \eqref{finalE} it suffices to show
\begin{align}
 R^{-C\delta} \lambda_1 \rho_2 \# \mathcal
  W_1^{ \nsim B}(q; \nu_1, \nu_2'))
 \lesssim   \#\mathbf S   \lesssim
   R^{C\delta} \#\mathcal W_2 \label{LUB}.
\end{align}
For the  lower bound  it is enough to show that,
for each  $\omega_1 \in  \mathcal W_1^{\nsim
B}(q; \nu_1, \nu_2')$,
\[
  \#\{ (\widetilde q, \omega_2) \in \mathcal Q(\rho_1,\rho_2) \times \mathcal W_2 :\,\pi_{\omega_1} \cap R^{\delta}\widetilde q \neq
  \emptyset,\,\pi_{\omega_2} \cap R^{\delta}\widetilde q \neq
  \emptyset,\, \mathrm{dist}(\widetilde q,q) \ge R^{1-\delta} \} \ge R^{-C\delta}
  \lambda_1 \rho_2.
\]

By  \eqref{Treduc} $\omega_1$ contains as many as  $O(
\lambda_1)$ cubes $\widetilde q$ in $\mathcal Q(\rho_1,\rho_2)$.\footnote{Recall that  we are assume assuming   $q\in \mathcal Q(\rho_1, \rho_2)$, $\omega_1 \in  \mathcal W_1[\lambda_1;\rho_1, \rho_2]$ and $\omega_2 \in  \mathcal W_2[\lambda_1;\rho_1, \rho_2]$.}
Let $B^*(\omega_1) \in \mathcal Q$ be the cube which maximizes the quantity given by
\eqref{B*max} with $\omega=\omega_1$. Since $\omega_1 \nsim B$,
it follows from
the definition of the relation $\sim$ that  $\mathrm{dist}(B^*(\omega_1),B) \gtrsim  R^{1-\delta}$.
Since $\pi_{\omega_1} +O(R^{\frac12+\delta})$ can be covered by $R^{C\delta}$ cubes $B$,  by a simple pigeonholing argument we get
\[
  \#\{ \widetilde q \in \mathcal Q(\rho_1,\rho_2) :~\pi_{\omega_1} \cap R^{\delta}\widetilde q \neq
  \emptyset,~ \mathrm{dist}(\widetilde q,q) \ge   R^{1-\delta} \}
  \gtrsim R^{-C\delta}  \lambda_1.
\]

Next, for the upper bound it suffices to show that, for any $\omega_2 \in
\mathcal W_2$,
\begin{equation}
\label{upper}
\begin{aligned}
\# \{(\widetilde q,\omega_1) \in \mathcal Q( \rho_1, \rho_2) \times
  &\mathcal W_1^{ \nsim B}(q; \nu_1, \nu_2') :
  \pi_{\omega_1} \cap R^{\delta}\widetilde q \neq
  \emptyset, \\ & \,\pi_{\omega_2} \cap R^{\delta}\widetilde q \neq
  \emptyset,\ \ \mathrm{dist}(\widetilde q,q) \gtrsim R^{1-\delta} \}
  \lesssim  R^{C\delta}.
\end{aligned}
\end{equation}
Let $\mathbf z_0$ be the center of $q$. Then,  by the definition of
$\mathcal W_1^{\nsim B}(q_0; \nu_1, \nu_2')$,  it follows that
\[
\bigcup_{\omega_1 \in  \mathcal W_1^{\nsim B}(q; \nu_1, \nu_2')}  \pi_{\omega_1} \subset \Gamma_1^{\nu_1,\nu_2'}( CR^{\frac12+\delta}) +\mathbf z_0.
\]
If  $\omega_2 \in  \mathcal W_2$, then it follows from Lemma \ref{platetrans} that the intersection
\[
\pi_{\omega_2} \cap \Big(\bigcup_{\omega_1 \in  \mathcal W_1^{\nsim B}(q; \nu_1, \nu_2')}  \pi_{\omega_1} \Big)
\] is contained in a cube  of sidelength $O(R^{1/2+\delta})$.
Thus,  there are at most $O(R^{C\delta})$ choices of
balls $\widetilde q \in \mathcal Q(\rho_1,\rho_2)$ such that  $(\widetilde q,\omega_1)$ is contained in the set in \eqref{upper}.
On the other hand, since
$\mathrm{dist}(\widetilde q,q) \gtrsim R^{1-C\delta}$,  we have
\begin{equation} \label{distantcubes}
\#\{ w_1\in \mathcal W_1^{\nsim B}(q; \nu_1, \nu_2') :    \pi_{\omega_1} \cap R^{\delta} \widetilde q \neq \emptyset, \quad  \pi_{\omega_1} \cap
R^{\delta}q\neq \emptyset    \} \lesssim  R^{C\delta}.
\end{equation}
To see this, by scaling it is enough to check that the map $  S_1\ni \nu: \to \sum_{i=1}^k t_j \nabla\varphi_i(\nu)$ is {one-to-one} whenever $|t|=1$.  But this follows from the condition \eqref{hessianinvertible} if we take $S_1$ to be small enough.
 Thus we obtain the claim \eqref{upper}. Hence, we also have \eqref{redMa2}, which finishes the proof of Proposition \ref{prop}. This completes the proof of  Theorem \ref{thm}.
\qed

\begin{proof} [Proof of Theorem \ref{cor2}]  Thanks to Lemma  \ref{platetrans2}, the line  of argument in the proof of Theorem \ref{thm} works without modification except that we need to show \eqref{distantcubes}. However, to prove  \eqref{distantcubes} we don't need to show  $S_1\ni \nu \mapsto \sum_{i=1}^k t_j \nabla\varphi_i(\nu)$ is one-to-one. Instead, as is clear after rescaling
it is enough to show that $  \Pi^{\nu_1, \nu_2'}\ni \nu \mapsto \sum_{i=1}^k t_j \nabla\varphi_i(\nu)$ is one-to-one.  Let $\mathbf t_1, \dots, \mathbf t_{d-k}$ be a set of vectors spanning  the tangent space of $\Pi^{\nu_1, \nu_2'}$ at $\nu_0$.  Then the above follows if we  show that the matrix
\[ ( \mathbf t_1^t ,  \dots, \mathbf t_{d-k}^t) \Big( \sum_{i=1}^k t_j H\varphi_i(\nu_0) \Big) \]
has rank $d-k$ for $|t|=1$.  In fact, $\mathbf t_1, \dots, \mathbf t_{d-k}$  are almost normal to the span of $\{\nabla \varphi_i(\nu_2) -  \nabla \varphi_i(\nu_1): i=1,\dots, k\}$.  These vectors are close to  $\mathbf n_1, \dots \mathbf n_{d-k}$.
Hence, assuming that $S_1$ and $S_2$ are small enough, the above follows if we show
$\mathbf N(\nu_2, \nu_1) \sum_{i=1}^k t_j H\varphi_i(\nu_0)$ has rank $d-k$. This clearly follows from  \eqref{normalcondition}.
\end{proof}


\section{Restriction  estimates for complex surfaces}

In this section we provide the proofs of Corollary \ref{prop:Gercomp} and Theorem \ref{thm:cpl}.  In what follows we set $k=2$, $d=2n$.

\begin{proof}[Proof of Corollary \ref{prop:Gercomp}]
\newcommand{\Rp}{\Re^\ast}
Let $\varphi_1$, $\varphi_2$ be given by
$\frac12 z^t D z=\varphi_1+ i\varphi_2$ so that
\[  \varphi_1(x,y)= \frac12  \big(x^t D x- y^t Dy),\,\,\, \varphi_2(x,y)= x^t Dy,\quad (x,y)\in \mathbb R^n\times \mathbb R^n.\]
In order to prove Corollary \ref{prop:Gercomp}  we need only to show that the condition \eqref{ssc} implies
the assumptions in Theorem \ref{thm}.

Let us set $z_j=x_j+i y_j\in \mathbb C^{n}$ for $j=1,2$,    $\delta_x=x_2- x_1$,  and $\delta_y=y_2- y_1$.
Then a computation shows that  the associated matrix  $\mathbf M(t, z_1, z_2, z)$ is given by
\[\mathbf M(t, z_1, z_2, z)=
\begin{pmatrix}
 &0&0& \delta_x^t D& -\!\!\!\!\!\!&\delta_y^t D
 \\
& 0&0& \delta_y^t D&   &\delta_x^t D
 \\
&D \delta_x& D \delta_y& t_1D&  &t_2 D
\\
-\!\!\!\!\!\!&D\delta_y & D \delta_x&   t_2 D&   -\!\!\!\!\!\! &t_1D
\end{pmatrix}\,.
\]
Note that
\[  \sum_{j=1}^2 t_j H \varphi_j =  \begin{pmatrix}
t_1D& t_2 D
\\
 t_2 D&   -t_1D
\end{pmatrix}. \]
Then, it is easy to see that the inverse of $\sum_{j=1}^2 t_j H \varphi_j $  is
$ {(t_1^2+t_2^2)^{-1}} \begin{pmatrix}  t_1 D^{-1}& t_2 D^{-1} \\  t_2 D^{-1}&   \,\,-t_1D^{-1} \end{pmatrix}$.
So, the assumption \eqref{hessianinvertible} holds. Hence, it suffices to show that   \eqref{ssc} implies \eqref{hessiannonvanishing}. By the block matrix formula
we only need  to check
\[
\det  \left[
 \begin{pmatrix} \delta_x^t D& -\delta_y^t D
 \\
\delta_y^t D&  \delta_x^t D
 \end{pmatrix}
\begin{pmatrix}
 t_1 D^{-1}& t_2 D^{-1}
\\
 t_2 D^{-1}&   -t_1D^{-1}
\end{pmatrix}
\begin{pmatrix} D \delta_x& D \delta_y
\\
-D\delta_y & D \delta_x \end{pmatrix} \right] \neq 0.\]
By a direct computation it is not difficult to see  that the left-hand side equals
\[ -(t_1^2+ t_2^2)\big( \big( \delta_x^tD\delta_x-\delta_y^t D\delta_y\big)^2+ 4 \big(\delta_x^tD\delta_y\big)^2 \,  \big)
 \footnote{In fact, the product of the three matrices  is equal to  $
 \begin{pmatrix}
t_1& -t_2
\\
 t_2 &  t_1
 \end{pmatrix}
 \begin{pmatrix}
\delta_x^tD\delta_x-\delta_y^t D\delta_y & 2 \delta_x^tD\delta_y
\\
2 \delta_x^tD\delta_y &  \delta_y^t D\delta_y-\delta_x^tD\delta_x
\end{pmatrix} .$ } \]
Since $(z_2-z_1)^t D(z_2-z_1)= \delta_x^tD\delta_x-\delta_y^t D\delta_y+2i \delta_x^tD\delta_y$, it is now clear that   \eqref{ssc} implies \eqref{hessiannonvanishing}.
\end{proof}

\begin{proof}[Proof of Theorem \ref{thm:cpl}] From the bilinear estimate we can get the linear estimate by adapting the arguments in \cites{TVV, V, L}.
Since $D$ is nonsingular and symmetric, {{} by making use of linear transforms we may assume that
\[  D = \begin{pmatrix}
1&0 \\
0& \pm 1
\end{pmatrix}, \]
and so we have either $\Phi(z_1,z_2) = z_1^2 + z_2^2 = (z_1+ i z_2)(z_1- i z_2)$ or
$\Phi(z_1,z_2) = z_1^2 - z_2^2 = (z_1+ z_2)(z_1- z_2)$.
By a linear change of variables {the problem can be} further reduced to showing Proposition \ref{thm:cpl} when $\Phi(z_1,z_2) = z_1z_2$.}

The following is an immediate consequence of Theorem \ref{thm} and the translation invariance of the bilinear estimate.

\begin{lemma} \label{lem:bilAA}
Let $\Phi(z_1,z_2) = z_1 z_2$  and  $Q_1, Q_2\subset \bbC^2$ be closed cubes.
Assume that %
\[  2^4\ge |z_1 - w_1| \ge 2^{-1} , \text{ and }  2^4\ge |z_2 - w_2| \ge 2^{-1}
\]
whenever $(z_1, z_2) \in Q_1$ and $(w_1, w_2) \in Q_2$.
If $\supp(f) \subset Q_1$ and $\supp(g) \subset Q_2$, then  for $q > \frac{10}{3}$ and $\frac{1}{p} + \frac{5}{3q} <1$
\[ \|  Ef \,  Eg \|_{q/2} \le C_{p,q}\, \| f\|_p \| g\|_p.
\]
\end{lemma}

In the next lemma the hypothesis of `nonvanishing rotational curvature' is weakened to the usual separation condition. But then, for the conclusion to hold, the pair $(1/p, 1/q)$ needs to satisfy a more restrictive condition. This lemma is an analog of Proposition 4.1 in \cite{L}. 

\begin{lemma} \label{separationonly} Let $Q_1$, $Q_2$ be closed cubes  in $\bbC^2$ such that  $\dist(Q_1, Q_2) \ge  1$. If $\supp(f) \subset Q_1$ and $\supp(g) \subset Q_2$, then there is a constant $C_{p,q}$ such that
\[ \| Ef \, Eg \|_{q/2} \le C_{p,q}\, \| f\|_p \| g\|_p
\]
if $\frac{1}{p} + \frac{2}{q} < 1$, $q > \frac{10}{3}$, or $\|  E \chi_F\, E\chi_G \|_{q/2} \lc \| f\|_{p,1} \| g\|_{p,1}$
if $\frac{1}{p} + \frac{2}{q} = 1$, $q > \frac{10}{3}$.\end{lemma}
By translation it is clear that in Lemma \ref{lem:bilAA} and Lemma \ref{separationonly} the same estimate holds with $Q_1$, $Q_2$  replaced by  $Q_1+a$, $Q_2+a$, respectively, for any $a\in \mathbb C^2$.
It is possible to prove the strong-type estimate $\|  E \chi_F\, E\chi_G \|_{q/2} \lc \| f\|_{p} \| g\|_{p}$ for $\frac{1}{p} + \frac{2}{q} = 1$, $q > \frac{10}{3}$ by making use of the asymmetric  estimates  which are obtained in the course of proof of Proposition \ref{prop:endEst} and the bilinear interpolation (see e.g. \cite{BeL}*{Sec. 3.13, 5(b)}). However, we have decided not to include the details here, because it does not seem to have any consequence{s} for linear estimates.

\begin{proof}[Proof of Lemma \ref{separationonly}] By interpolation it suffices to consider the case $\frac{10}3< q\le 4$ and $p\le q$. By decomposition of the domains,
followed by translation and scaling, we may assume that
$Q_1 = H_1 \times K$ and $Q_2 = H_2 \times K$, where
$\dist(H_1, H_2) \ge 2^{-1}$ and $K$ is the unit cube in $\bbC$, centered at the origin.

By a Whitney decomposition, we get
\[ (K \times K) \setminus D = \bigcup_{j>1} \bigcup_{ (k, k'):  I^j_{k} \sim I^j_{k'} }  I^j_{k} \times I^j_{k'}
\]
where $D = \{ (z_2, w_2): ~ z_2 = w_2 \}$, and $\{I^j_k\}_{k}$ are the dyadic cubes in $\bbC$ of sidelength $2^{-j}$, and as usual the notation $ I^j_{k}\sim  I^j_{k' }$ means that the parent cubes of $ I^j_{k}$ and $ I^j_{k'} $ are adjacent, while $ I^j_{k}$ and $ I^j_{k'}$ are not.

Let us set
\[ f^j_k (z_1, z_2) = \chi_{I^j_k} (z_2) f (z_1, z_2), \quad g^j_k(w_1, w_2) = \chi_{I^j_k} (w_2) g (w_1, w_2).\]
Then, since the cubes $I^j_{k} \times I^j_{k'} $ are almost disjoint, we may write
\[ Ef\, Eg = \sum_j \sum_{(k, k'):  I^j_{k} \sim I^j_{k'} } E(f^j_k)\, E(g^j_{k'}) .
\]
Since $ q> \frac{10}{3}$,   we get
\begin{align*}
\| Ef\, Eg \|_{q/2} &\le  \sum_j \| \sum_{I^j_{k} \sim I^j_{k'} } E(f^j_k) \, E(g^j_{k'}) \|_{q/2}
  \lc  \sum_j \Big( \sum_{I^j_{k} \sim I^j_{k'} } \| E(f^j_k) \, E(g^j_{k'}) \|_{q/2}^{q/2} \Big)^{2/q},
\end{align*}
where the last inequality follows from Lemma 6.1 in \cite{TVV}. Here, we used the fact that for each fixed $j$ the supports of the Fourier transforms of  ${E(f^j_k) \, E(g^j_{k'})}$ have uniformly bounded overlap as $(k, k')$ varies, provided that  $I^j_{k} \sim I^j_{k'}$. This  is a consequence of  the Whitney decomposition.
We now claim that
if  $I^j_{k} \sim I^j_{k'}$, then
\begin{equation}\label{scaled}
\| E(f^j_k) \, E(g^j_{k'}) \|_{q/2} \lc 2^{4j \big(\frac{1}{p}+\frac{2}{q}-1 \big)} \| f^j_k\|_p \, \| g^j_{k'} \|_{p}
\end{equation}
 when $\frac{1}{p} + \frac{5}{3q} <1$, $ q> \frac{10}{3}$.   This is  an easy consequence of a translated version of Lemma \ref{lem:bilAA}.  Assuming this for the moment, we will finish the proof.
Since $q\ge p$,  for  $\frac{1}{p} + \frac{5}{3q} <1$, $4>q> \frac{10}{3}$, we have
\begin{align*}
&\qquad \| Ef\, Eg \|_{q/2} \le  \sum_j  2^{4j \big(\frac{1}{p}+\frac{2}{q}-1 \big)} \Big( \sum_{I^j_{k} \sim I^j_{k'} } \| f^j_k\|_p^{q/2}  \| g^j_{k'} \|_{p}^{q/2} \Big)^{2/q}
\\ & \lc  \sum_j  2^{4j \big(\frac{1}{p}+\frac{2}{q}-1 \big)} (\sum_{I^j} \| f^j_k\|_p^{p} )^{1/p}  (\sum_{J^j} \| g^j_{k'} \|_{p}^{p} )^{1/p}
 \lc  \sum_j  2^{4j \big(\frac{1}{p}+ \frac{2}{q}-1 \big)} \| f\|_p\,  \| g \|_{p}.
\end{align*}

Now take $f = \chi_F$ and $g = \chi_G$ for measurable sets $F$, $G$ contained in $V_1$, $V_2$, respectively.
Fix $p$, $q$ with $4>q > \frac{10}{3}$, $\frac{1}{p} + \frac{2}{q} = 1$, and choose $p_1$ and $p_2$  such that
$\frac{1}{p_j} + \frac{5}{3q} <1$, $j=1,2$, and
\[ \frac{1}{p_1} + \frac{2}{q} -1 = \eta, \quad  \frac{1}{p_2} + \frac{2}{q} -1 = - \eta
\]
for some small $\eta >0$. Then by applying the last estimate for $p = p_1$ and $p=p_2$,  we obtain
\begin{align*}
\|  E \chi_F\, E\chi_G \|_{q/2}  & \lc \sum_j \min \{ 2^{4j\eta} |F|^{\eta +1 - 2/q} |G|^{\eta +1 - 2/q}, ~ 2^{-4j\eta} |F|^{-\eta +1 - 2/q} |G|^{-\eta +1 - 2/q} \}
\\
&\lc |F|^{1-2/q} |G|^{1-2/q} = |F|^{1/p} |G|^{1/p}.
\end{align*}
This shows the estimate $\|  E \chi_F\, E\chi_G \|_{q/2} \lc \| f\|_{p,1} \| g\|_{p,1}$
for $\frac{1}{p} + \frac{2}{q} = 1$, $4>q> \frac{10}{3}$.

\

Now it remains to show \eqref{scaled}.  Clearly, $I^j_{k}$ and $ I^j_{k'}$ are contained in a ball of radius $2^{2-j}$ and $\dist(I^j_{k}, I^j_{k'})\ge 2^{-1-j}$. Hence, by {a change of variables,}
\begin{align*}
E(f^j_k) (w) &= 2^{-2j} \, E ( f^j_k (\cdot, 2^{-j}\,\cdot )) (w_1 ,  2^{-j} w_2,  2^{-j}  w_3),  \\
 E(g^j_{k'}) (w) &= 2^{-2j} \, E ( g^j_{k'} (\cdot, 2^{-j} \, \cdot )) (w_1 ,  2^{-j}  w_2, 2^{-j}   w_3).\end{align*}
Then we see that $\supp f^j_k (\cdot, 2^{-j}\,\cdot )\subset H_1\times \widetilde  I_1$ and
$g^j_{k'} (\cdot, 2^{-j} \, \cdot )\subset H_2\times \widetilde  I_2$ if $ \dist(\widetilde  I_1, \widetilde  I_2)\ge 2^{-1}$ and $\widetilde  I_1, \widetilde  I_2$ are contained in a ball of radius $\le 2^{3}$.  The assumption of Lemma \ref{lem:bilAA}  is satisfied with $f= f^j_k (\cdot, 2^{-j}\,\cdot )$ and $g=g^j_{k'} (\cdot, 2^{-j} \, \cdot )$.   Hence   we may apply it to  $E ( f^j_k (\cdot, 2^{-j}\,\cdot )) E ( f^j_k (\cdot, 2^{-j}\,\cdot ))$ and get \eqref{scaled}.
This completes the proof.
\end{proof}

Once Lemma \ref{separationonly} is established, the usual argument in \cite{TVV}, used to deduce linear estimates from bilinear ones, works without modification. We omit the details.
\end{proof}


\medskip

\noindent{\bf Acknowledgements.}
We would like to thank Andreas Seeger and the anonymous referee  for bringing the reference \cite{A} and \cite{B} to our attention.


\end{document}